\documentclass[11pt,a4paper,reqno]{amsart}
\usepackage[T1]{fontenc}
\usepackage[utf8]{inputenc}
\usepackage{moreenum}

\usepackage{lmodern}
\usepackage{microtype}

\usepackage[english]{babel}

\usepackage{amsthm}
\usepackage{thmtools}
\usepackage{amsmath}
\usepackage{amsfonts} 
\usepackage{amssymb}
\usepackage{mathrsfs}
\usepackage{enumitem}

\usepackage{esint}

\usepackage[hcentering]{geometry}
\usepackage[unicode=true, pdflang={en-GB}, pdfusetitle,hypertexnames=false]{hyperref}
\hypersetup{allbordercolors=[rgb]{0.3,0.4,1}}

\let\C\undefined

\usepackage[abbrev,backrefs]{amsrefs}
\usepackage[capitalize]{cleveref}

\numberwithin{equation}{section}

\newtheorem{proposition}{Proposition}[section]

\newtheorem{theorem}[proposition]{Theorem}
\newtheorem{lemma}[proposition]{Lemma}

\theoremstyle{definition}

\newcommand{\defeq}{\coloneqq}

\newcommand{\Nset}{\mathbb{N}}
\newcommand{\Zset}{\mathbb{Z}}
\newcommand{\Rset}{\mathbb{R}}

\newcommand{\Sset}{\mathbb{S}}
\newcommand{\Qset}{\mathbb{Q}}
\newcommand{\Bset}{\mathbb{B}}
\newcommand{\tangent}{\mathrm{T}}
\newcommand{\sobolev}{\smash{\mathrm{W}}}
\newcommand{\soboleh}{\smash{\mathrm{H}}}
\newcommand{\smooth}{\mathrm{C}}
\newcommand{\lebesgue}{\mathrm{L}}
\newcommand{\continuous}{\mathrm{C}}
\newcommand{\VMO}{\mathrm{VMO}}
\newcommand{\BMO}{\mathrm{BMO}}

\newcommand{\dif}{\,\mathrm{d}}
\usepackage{mathtools}

\newcommand{\compose}{\,\circ\,}
\newcommand{\manifold}[1]{\mathcal{#1}}

\DeclarePairedDelimiter{\brk}{(}{)}
\DeclarePairedDelimiter{\sqb}{[}{]}
\DeclarePairedDelimiter{\abs}{\lvert}{\rvert}
\DeclarePairedDelimiter{\norm}{\lVert}{\rVert}
\DeclarePairedDelimiter{\seminorm}{\lvert}{\rvert}
\DeclarePairedDelimiter{\floor}{\lfloor}{\rfloor}

\DeclarePairedDelimiterX{\intvc}[2]{[}{]}{#1,#2}
\DeclarePairedDelimiterX{\intvl}[2]{(}{]}{#1,#2}
\DeclarePairedDelimiterX{\intvr}[2]{[}{)}{#1,#2}
\DeclarePairedDelimiterX{\intvo}[2]{(}{)}{#1,#2}
\DeclarePairedDelimiterX{\setcond}[2]{\{}{\}}{#1 \,\delimsize\vert\, #2}

\newcommand{\restr}[1]{\vert_{#1}}
\newcommand{\Deriv}{\mathrm{D}}
\newcommand{\extdiff}{\mathrm{d}}

\newcommand{\eqpunct}[1]{\,\text{#1}}

\newcommand\stSymbol[1][]{%
\nonscript\;#1\vert
\allowbreak
\nonscript\;
\mathopen{}}
\DeclarePairedDelimiterX\set[1]\{\}{%
\renewcommand\st{\stSymbol[\delimsize]}
#1
}
\providecommand{\st}{\stSymbol}

\DeclareMathOperator{\dist}{dist}

\DeclareMathOperator{\inj}{inj}

\DeclareMathOperator{\tr}{tr}
\newcommand{\weakto}{\rightharpoonup}

\usepackage{constants}

\renewcommand{\PrintDOI}[1]{%
  \href{http://dx.doi.org/#1}{doi:#1}%
}
\DefineSimpleKey{bib}{arxiv}
\BibSpec{book}{%
    +{}  {\PrintPrimary}                {transition}
    +{,} { \textit}                     {title}
    +{.} { }                            {part}
    +{:} { \textit}                     {subtitle}
    +{,} { \PrintEdition}               {edition}
    +{}  { \PrintEditorsB}              {editor}
    +{,} { \PrintTranslatorsC}          {translator}
    +{,} { \PrintContributions}         {contribution}
    +{,} { }                            {series}
    +{,} { \voltext}                    {volume}
    +{,} { }                            {publisher}
    +{,} { }                            {organization}
    +{,} { }                            {address}
    +{,} { \PrintDateB}                 {date}
    +{,} { }                            {status}
    +{,} { \PrintDOI}                   {doi}
    +{}  { \parenthesize}               {language}
    +{}  { \PrintTranslation}           {translation}
    +{;} { \PrintReprint}               {reprint}
    +{.} { }                            {note}
    +{.} {}                             {transition}
    +{}  {\SentenceSpace \PrintReviews} {review}
}
\newcommand{\PrintarXiv}[1]{%
  \href{https://arxiv.org/abs/#1}{arXiv:#1}%
}
\BibSpec{arxiv}{%
    +{}  {\PrintPrimary}                {transition}
    +{,} { \textit}                     {title}
    +{.} { }                            {part}
    +{:} { \textit}                     {subtitle}
    +{,} { \PrintDateB}                 {date}
    +{,} { \PrintarXiv}                 {arxiv}
    +{}  { \parenthesize}               {language}
    +{.} { }                            {note}
    +{.} {}                             {transition}
}

\title
{%
Equi-integrable approximation of Sobolev mappings between manifolds%
}

\author{Jean Van Schaftingen}


\newcommand{\sectref}[1]{\hyperref[#1]{\S \ref*{#1}}}

\begin{document}

\address{
Universit\'e catholique de Louvain, Institut de Recherche en Math\'ematique et Physique, Chemin du Cyclotron 2 bte L7.01.01, 1348 Louvain-la-Neuve, Belgium}

\email{Jean.VanSchaftingen@UCLouvain.be}

\begin{abstract}
We show that limits of sequences of smooth maps between compact Riemannian manifolds with equi-integrable $\mathrm{W}^{1, p}$-Sobolev energy can always be strongly approximated by smooth maps, giving a counterpart of Hang's density result in $\mathrm{W}^{1, 1}$ for the Sobolev space $\mathrm{W}^{1, p}$ with integer $p \ge 2$.
Our result extends to higher-order Sobolev spaces and is straightforward in fractional Sobolev spaces.
We also provide a proof based on the weak continuity of Jacobians in the cases where the cohomological criterion of Bethuel, Demengel, Coron and H\'elein applies.
\end{abstract}

\thanks{The author was supported by the Mandat d'Impulsion Scientifique F.4523.17, ``Topological singularities of Sobolev maps'' and by the Projet de Recherche T.0229.21 ``Singular Harmonic Maps and Asymptotics of Ginzburg--Landau Relaxations'' of the Fonds de la Recherche Scientifique--FNRS}

\maketitle

\setcounter{tocdepth}{1}
\tableofcontents
\setcounter{tocdepth}{5}

\section{Introduction}
Given two compact Riemannian manifolds \(\manifold{M}\) and \(\manifold{N}\) and \(p \in \intvr{1}{\infty}\),
we are interested in the \emph{Sobolev space of mappings}
\begin{equation}
\label{eq_Roejiejai2zah4iek0auwae0}
\sobolev^{1, p}\brk{\manifold{M}, \manifold{N}}
 = \set{u \in \sobolev^{1, p}\brk{\manifold{M}, \Rset^\nu} \st  u \in \manifold{N} \text{ almost everywhere in \(\manifold{M}\)} }\eqpunct{,}
\end{equation}
where \(\manifold{N}\) is isometrically embedded into \(\Rset^\nu\) -- without loss of generality in view of the Nash embedding theorem  \cite{Nash_1956} -- and where the classical linear Sobolev space is defined as 
\[
 \sobolev^{1, p}\brk{\manifold{M}, \Rset^\nu}
 \defeq \set[\Big]{u \in \lebesgue^p \brk{\manifold{M}, \Rset^\nu} \st  u \text{ is weakly differentiable and } \int_{\manifold{M}} \abs{\Deriv u}^p < \infty }\eqpunct{.}
\]
Sobolev spaces of mappings appear naturally in calculus of variations and partial differential equations in many contexts, including the theory of harmonic maps \cite{Eells_Lemaire_1978}, physical models with non-linear order parameters \cite{Mermin1979}, Cosserat models in elasticity \citelist{\cite{Ericksen_Truesdell_1958}}, and the description of attitudes of a square or a cube in computer graphics \cite{Huang_Tong_Wei_Bao_2011}.

A natural subset of \( \sobolev^{1, p}\brk{\manifold{M}, \manifold{N}}\) is the class of \emph{strongly approximable mappings,} defined as the closure with respect to the norm of smooth mappings:
\begin{multline}
\label{eq_cho3eibic6eo6ieSh9iPhieh}
 \soboleh^{1, p}_{\mathrm{St}} \brk{\manifold{M}, \manifold{N}}
 \defeq \Bigl\{u \in \sobolev^{1, p} \brk{\manifold{M}, \manifold{N}}
 \st  \lim_{n \to \infty} \int_{\manifold{M}} \abs{u_n - u}^p + \abs{\Deriv u_n - \Deriv u}^p = 0\\ \text{ and } u_n \in \smooth^\infty \brk{\manifold{M}, \manifold{N}}\Bigr\}
 \eqpunct{.}
\end{multline}
The \emph{strong approximation problem} asks whether one has
\begin{equation}
\label{eq_liiy1mai7jeit6Athiugie4y}
 \soboleh^{1, p}_{\mathrm{St}} \brk{\manifold{M}, \manifold{N}}
 = \sobolev^{1, p} \brk{\manifold{M}, \manifold{N}}
 \eqpunct{,}
\end{equation}
that is, whether any map in \(\sobolev^{1, p} \brk{\manifold{M}, \manifold{N}}\) is the strong limit in \(\sobolev^{1, p}\brk{\manifold{M}, \Rset^\nu}\) of smooth maps taking their values in \(\manifold{N}\).
The classical argument in the linear case of convolution by a family of mollifiers is not adapted to the nonlinear setting of manifolds, since there is no reason for an average of points in the target manifold to still satisfy the same constraint.

When \(p \ge \dim \manifold{M}\), Schoen and Uhlenbeck have shown that the embedding of \( \sobolev^{1, p} \brk{\manifold{M}, \manifold{N}}\) into the set of continuous mappings \(\continuous \brk{\manifold{M}, \manifold{N}}\), for \(p > \dim \manifold{M}\), or into the set of vanishing mean oscillation mappings \(\VMO \brk{\manifold{M}, \manifold{N}}\), for \(p = \dim \manifold{M}\), ensures that \eqref{eq_liiy1mai7jeit6Athiugie4y} holds \cite{Schoen_Uhlenbeck_1983}.
When \(1 \le p < \dim \manifold{M}\), the question is more delicate.
If \(\pi_1 \brk{\manifold{M}} \simeq \dotsb \simeq \pi_{\floor{p - 1}}\brk{\manifold{M}} \simeq \set{0}\), which includes the case of the \(m\)-dimensional sphere \(\manifold{M} = \Sset^m\), then
\eqref{eq_liiy1mai7jeit6Athiugie4y} holds if and only if \(\pi_{\floor{p}} \brk{\manifold{N}} \simeq \set{0}\)  \citelist{\cite{Bethuel_1991}\cite{Hang_Lin_2003_II}}, where \(\floor{p} \in \Nset\) is the integer part of \(p\) characterised by the condition \(p - 1<\floor{p} \le p\) and where \(\pi_{\ell} \brk{\manifold{N}}\) denotes for \(\ell \in \Nset \setminus \set{0}\) the \(\ell\)-th homotopy group of \(\manifold{N}\) (see for example \cite{Hatcher_2002}); the condition \(\pi_{\ell}\brk{\manifold{N}} \simeq \set{0}\) is equivalent to the fact that every map in \(\smooth \brk{\Sset^\ell, \manifold{N}}\) is the restriction to \(\Sset^\ell = \partial \Bset^{\ell + 1}\) of a map in \(\smooth \brk{\Bset^{\ell + 1}, \manifold{N}}\).
For a general compact domain manifold \(\manifold{M}\),  \eqref{eq_liiy1mai7jeit6Athiugie4y} holds if and only if every mapping in \(\continuous \brk{\smash{\manifold{M}^{\floor{p}}}, \manifold{N}}\) is the restriction of a map in \(\smooth \brk{\manifold{M}, \manifold{N}}\), where \(\smash{\manifold{M}^\ell}\) is the \(\ell\)-dimensional component of a triangulation of the domain \(\manifold{M}\) \cite{Hang_Lin_2003_II}.

When \eqref{eq_liiy1mai7jeit6Athiugie4y} fails, one can wonder whether smooth maps are dense for a weaker notion of convergence.
The \emph{bounded approximation problem} consists in studying the \emph{bounded sequential closure of smooth maps} defined as 
\begin{multline*}
 \soboleh^{1, p}_{\mathrm{Bd}} \brk{\manifold{M}, \manifold{N}}
 \defeq \Bigl\{u \in \sobolev^{1, p}\brk{\manifold{M}, \manifold{N}}
 \st
 \lim_{n \to \infty} \int_{\manifold{M}} \abs{u_n - u}^p = 0\eqpunct, \sup_{n \in \Nset} \int_{\manifold{M}} \abs{\Deriv u_n}^p <\infty\\ \text{ and } u_n \in \smooth^\infty \brk{\manifold{M}, \manifold{N}}\Bigr\}\eqpunct.
\end{multline*}
The definitions imply immediately the inclusions
\begin{equation}
\label{eq_ohv3chieyaeNg2uuroYeish4}
 \soboleh^{1, p}_{\mathrm{St}} \brk{\manifold{M}, \manifold{N}}\subseteq
 \soboleh^{1, p}_{\mathrm{Bd}} \brk{\manifold{M}, \manifold{N}}
 \subseteq \sobolev^{1, p} \brk{\manifold{M}, \manifold{N}}
 \eqpunct{,}
\end{equation}
so that the bounded sequential closure of smooth maps
\(\soboleh^{1, p}_{\mathrm{Bd}} \brk{\manifold{M}, \manifold{N}}\) is only interesting when the strong approximation property \eqref{eq_liiy1mai7jeit6Athiugie4y} fails.

When \(1 \le p < \dim \manifold{M}\) and when \(p \not \in \Nset\), Bethuel has proved the first inclusion in \eqref{eq_ohv3chieyaeNg2uuroYeish4} becomes an equality \cite{Bethuel_1991}:
\begin{equation}
\label{eq_saeNohshai4pheikiechai7Z}
 \soboleh^{1, p}_{\mathrm{St}}\brk{\manifold{M}, \manifold{N}}
 = \soboleh^{1, p}_{\mathrm{Bd}} \brk{\manifold{M}, \manifold{N}}
 \eqpunct{.}
\end{equation}
On the other hand if \(1 \le p < \dim \manifold{M}\) and if \(p \in \Nset\), then \eqref{eq_saeNohshai4pheikiechai7Z} holds if and only if \(\pi_p\brk{\manifold{N}}\simeq \set{0}\)  \citelist{\cite{Bethuel_1991}\cite{Hang_Lin_2003_II}}.
In the latter case, one can still wonder whether the second inclusion in \eqref{eq_ohv3chieyaeNg2uuroYeish4} is an equality;
this is the case when \(p = 1\) for which one has \cite{Hajlasz_1994} (see also \citelist{\cite{Hang_2002}\cite{Pakzad_Riviere_2003}\cite{Pakzad_2003}})
\begin{equation}
\label{eq_Buc0fa4quohmiem0ineeshae}
 \soboleh^{1, 1}_{\mathrm{Bd}}\brk{\manifold{M}, \manifold{N}}
 = \sobolev^{1, 1} \brk{\manifold{M}, \manifold{N}}
 \eqpunct;
\end{equation}
when \(p \ge 2\), one still has, if \(\pi_{1} \brk{\manifold{N}}\simeq \dotsb \simeq \pi_{p - 1}\brk{\manifold{N}} \simeq \set{0}\),
\begin{equation}
\label{eq_Ohl5ahkin7chooTieth1aisa}
 \soboleh^{1, p}_{\mathrm{Bd}}\brk{\manifold{M}, \manifold{N}}
 = \sobolev^{1, p} \brk{\manifold{M}, \manifold{N}}
 \eqpunct,
\end{equation}
but \eqref{eq_Ohl5ahkin7chooTieth1aisa} fails in general because of global topological obstructions \cite{Hang_Lin_2003_II} and local analytical obstructions \citelist{\cite{Bethuel_2020}\cite{Detaille_VanSchaftingen}}.

Closely related to the bounded sequential closure is the \emph{weak sequential closure of smooth maps}, defined as
\begin{multline*}
  \soboleh^{1, p}_{\mathrm{Wk}} \brk{\manifold{M}, \manifold{N}}
  \defeq
\bigl\{u \in \sobolev^{1, p} \brk{\manifold{M}, \manifold{N}} \st u_n \weakto u \text{ weakly in \(\sobolev^{1, p}\brk{\manifold{M}, \Rset^\nu}\)}\\\text{and } u_n \in \smooth^\infty\brk{\manifold{M}, \manifold{N}} \bigr\}\eqpunct.
\end{multline*}
The weak convergence is defined by the topology on \(\sobolev^{1, p} \brk{\manifold{M}, \manifold{N}} \) induced by the \emph{weak topology generated} on the linear space \(\sobolev^{1, p} \brk{\manifold{M}, \Rset^\nu} \) by its linear continuous forms.
In other words, \(u_n \weakto u \) weakly in \(\sobolev^{1, p} \brk{\manifold{M}, \Rset^\nu}\) as \(n \to \infty\) if and only if
for every linear continuous \(\xi \colon \sobolev^{1, p} \brk{\manifold{M}, \Rset^\nu} \to \Rset\), the sequence \(\brk{\xi \brk{u_n}}_{n \in \Nset}\) converges to \(\xi \brk{u}\) in \(\Rset\).
Even though it is not known whether the resulting topology on \(\sobolev^{1, p} \brk{\manifold{M}, \manifold{N}} \) depends on the isometric embedding \(\manifold{N}\subseteq \Rset^\nu\), the associated \emph{weak convergence} can be proved to be \emph{intrinsic}.

When \(p > 1\), it follows from characterisations of weak convergence that \(u_n \weakto u\) weakly in \( \sobolev^{1, p} \brk{\manifold{M}, \Rset^\nu} \) as \(n \to \infty\) if and only if \(u_n \to u\) strongly in  \(\lebesgue^{p} \brk{\manifold{M}, \manifold{N}}\) as \(n \to \infty\) and the sequence \(\brk{\int_{\manifold{M}} \abs{\Deriv u_n}^p}_{n \in \Nset}\) is bounded; one has therefore
\begin{equation}
\label{eq_ooNgooloshaefeigha8iehi1}
  \soboleh^{1, p}_{\mathrm{Wk}}\brk{\manifold{M}, \manifold{N}}
 = \soboleh^{1, p}_{\mathrm{Bd}} \brk{\manifold{M}, \manifold{N}}
 \eqpunct{.}
\end{equation}
On the other hand, when \(p = 1\), 
it follows from the Dunford-Pettis compactness criterion \citelist{\cite{Dunford_1939}\cite{Dunfold_Pettis_1940}\cite{Dieudonne_1951}} (see also \citelist{\cite{Royden_Fitzpatrick}*{\S 19.5}\cite{Dunford_Schwartz_1958}*{Thm.\ IV.8.9}\cite{Voigt_2020}*{Ch.\ 15}\cite{Brezis_2011}*{Thm.\ 4.30}\cite{Beauzamy_1982}*{\S VI.2}\cite{Diestel_1984}*{Ch.\ VII}\cite{Fonseca_Leoni_2007}*{Thm.\ 2.54}\cite{Bogachev_2007}*{Thm.\ 4.7.18}}) that
\(u_n \weakto u\) weakly in \( \sobolev^{1, 1} \brk{\manifold{M}, \manifold{N}} \) as \(n \to \infty\) if and only if \(u_n \to u\) in \(\lebesgue^1 \brk{\manifold{M}, \Rset^\nu}\) as \(n \to \infty\) and if the sequence \(\brk{\Deriv u_n}_{n \in \Nset}\) is \emph{equi-integrable} (or, equivalently, uniformly integrable):
\[
 \lim_{t \to \infty} \sup_{n \in \Nset}\int_{\manifold{M}} \brk{\abs{\Deriv u_n} - t}_+ = 0
 \eqpunct.
\]
It turns out that the bounded and weak sequential closures \(\soboleh^{1, 1}_{\mathrm{Wk}}\brk{\manifold{M}, \manifold{N}}\) and \(\soboleh^{1, 1}_{\mathrm{Bd}} \brk{\manifold{M}, \manifold{N}}\) differ when \(\pi_1 \brk{\manifold{N}} \not \simeq \set{0}\), that is, when the target manifold \(\manifold{N}\) is not simply-connected;
indeed one has then Hang's weak density result \cite{Hang_2002}
\begin{equation}
\label{eq_eipiekaiHohk3xeed9eijaso}
 \soboleh^{1, 1}_{\mathrm{Wk}}\brk{\manifold{M}, \manifold{N}}
 = \soboleh^{1, 1}_{\mathrm{St}} \brk{\manifold{M}, \manifold{N}}\eqpunct{;}
\end{equation}
together with \eqref{eq_Buc0fa4quohmiem0ineeshae} and the failure of \eqref{eq_liiy1mai7jeit6Athiugie4y}, this shows that \eqref{eq_ooNgooloshaefeigha8iehi1} then fails.

The goal of the present work is to explain the apparent contrast between  \eqref{eq_ooNgooloshaefeigha8iehi1} and \eqref{eq_eipiekaiHohk3xeed9eijaso}.
In order to do this, we consider the \emph{sequential equi-integrable closure} of smooth maps defined as
\begin{multline}
\label{eq_Thaedah9eeyohtheesho4aic}
 \soboleh^{1, p}_{\mathrm{Ei}} \brk{\manifold{M}, \manifold{N}}
 \defeq \Bigl\{u \in \sobolev^{1, p}\brk{\manifold{M}, \manifold{N}}
 \st
 \lim_{n \to \infty} \int_{\manifold{M}} \abs{u_n - u}^p = 0 \eqpunct,\\ \lim_{t \to \infty} \sup_{n \in \Nset} \int_{\manifold{M}} \brk{\abs{\Deriv u_n} - t}_+^p = 0 \text{ and } u_n \in \smooth^\infty \brk{\manifold{M}, \manifold{N}}\Bigr\}
 \eqpunct{.}
\end{multline}

It follows from the definitions and the properties of the corresponding notions of convergence that, when \(p > 1\), we have,
\begin{equation}
\label{eq_aepopa9oTahquaas6vah2aer}
 \soboleh^{1, p}_{\mathrm{St}} \brk{\manifold{M}, \manifold{N}}
 \subseteq
 \soboleh^{1, p}_{\mathrm{Ei}} \brk{\manifold{M}, \manifold{N}}
  \subseteq
 \soboleh^{1, p}_{\mathrm{Wk}} \brk{\manifold{M}, \manifold{N}}
 =
 \soboleh^{1, p}_{\mathrm{Bd}} \brk{\manifold{M}, \manifold{N}}
 \subseteq \sobolev^{1, p} \brk{\manifold{M}, \manifold{N}}\eqpunct{,}
\end{equation}
whereas, when \(p = 1\),
\begin{equation}
\label{eq_iaxei7oiZeghahfaiSaSooku}
 \soboleh^{1, 1}_{\mathrm{St}} \brk{\manifold{M}, \manifold{N}}
 \subseteq
 \soboleh^{1, 1}_{\mathrm{Ei}} \brk{\manifold{M}, \manifold{N}}
  =
 \soboleh^{1, 1}_{\mathrm{Wk}} \brk{\manifold{M}, \manifold{N}}
 \subseteq
 \soboleh^{1, 1}_{\mathrm{Bd}} \brk{\manifold{M}, \manifold{N}}
 \subseteq \sobolev^{1, 1} \brk{\manifold{M}, \manifold{N}}
 \eqpunct{.}
\end{equation}

Our main result states that equality holds in the first inclusion in both \eqref{eq_aepopa9oTahquaas6vah2aer} and \eqref{eq_iaxei7oiZeghahfaiSaSooku}. In other words, strong and equi-integrable approximation by smooth maps are always equivalent.

\begin{theorem}
\label{theorem_W1p_equiintegrable_strong}
If \(\manifold{M}\) and \(\manifold{N}\) are compact Riemannian manifolds and if \(p \in \intvr{1}{\infty}\), then
\begin{equation}
\label{eq_jee6ud9boog1DaMaut7nui3w}
  \soboleh^{1, p}_{\mathrm{St}} \brk{\manifold{M}, \manifold{N}}
 =
 \soboleh^{1, p}_{\mathrm{Ei}} \brk{\manifold{M}, \manifold{N}}\eqpunct{.}
\end{equation}
\end{theorem}

When either \(p \ge \dim \manifold{M}\) or \(p \not \in \Nset\), then \eqref{eq_saeNohshai4pheikiechai7Z} holds so that \eqref{eq_jee6ud9boog1DaMaut7nui3w} follows immediately from \eqref{eq_aepopa9oTahquaas6vah2aer};
\cref{theorem_W1p_equiintegrable_strong} is thus relevant only for \(p \in \set{1, \dotsc, \dim \manifold{M} - 1}\).

For \(p = 1\), \cref{theorem_W1p_equiintegrable_strong} is due to Hang \cite{Hang_2002}; \cref{theorem_W1p_equiintegrable_strong} shows that the formulation of Hang’s nonlinear result \eqref{eq_eipiekaiHohk3xeed9eijaso} in terms of sequential equi-integrable closure rather than the linear framework of sequential closure for the weak topology becomes a particular case of a general phenomenon rather than an exception.

Giaquinta, Modica and Souček have proved a counterpart of \cref{theorem_W1p_equiintegrable_strong} for the \emph{convergence of minors} \cite{Giaquinta_Modica_Soucek_1998}*{\S 3.4.1}.

Hang’s proof for \(p=1\) relies on Egorov’s theorem and the continuity of weakly differentiable mappings on a one-dimensional domain;
our argument in \S \ref{section_equiint} relies on techniques for homotopy of maps with vanishing mean oscillation due to Schoen and Uhlenbeck \cite{Schoen_Uhlenbeck_1983} and to Brezis and Nirenberg \cite{Brezis_Nirenberg_1995} combined with suitable Fubini and truncation arguments.

\Cref{theorem_W1p_equiintegrable_strong} is not an isolated result.
In \S \ref{section_higher_order}, we prove that it extends to \emph{higher-order Sobolev spaces} \(\sobolev^{k, p}\brk{\manifold{M}, \manifold{N}}\) with \(k \in \Nset\setminus \set{0}\)  and \(p \in \intvr{1}{\infty}\).
For \emph{fractional Sobolev spaces} \(\sobolev^{s, p} \brk{\manifold{M}, \manifold{N}}\) with \(s \in \intvo{0}{\infty}\setminus \Nset\) and \(p \in \intvr{1}{\infty}\) the equivalence also holds, but for quite trivial reasons, as pointwise convergent equi-integrable sequences converge strongly (see \cref{proposition_Wsp_strong_equi_integrable} below).

For some target manifolds, Bethuel, Coron, Demengel and Hélein have described the local obstructions to the strong approximation through a \emph{distributional Jacobian} \citelist{\cite{Bethuel_1990}\cite{Demengel_1990}} or more \emph{general cohomological criterion} \cite{Bethuel_Coron_Demengel_Helein_1991}.
We show in \S \ref{section_jacobian} that \cref{theorem_W1p_equiintegrable_strong} can then be deduced from these criteria and the \emph{equi-integrable sequential continuity of Jacobians,} generalizing Hang’s argument for the circle \cite{Hang_2002}*{First proof of Example 2.1}.

\section{First-order Sobolev mappings}
\label{section_equiint}

In order to prove \cref{theorem_W1p_equiintegrable_strong}, we will prove in this section that equi-integrable convergence in \(\sobolev^{1, p} \brk{\manifold{M}, \manifold{N}}\) preserves the homotopy on \(p\)-dimensional skeletons.

\subsection{Vanishing mean oscillation homotopies}
Given a compact manifold \(\manifold{M}\), the set of functions of \emph{vanishing mean oscillation} is defined as \cite{Brezis_Nirenberg_1995}
\begin{equation*}
 \VMO \brk{\manifold{M}, \Rset^\nu}
 \defeq
 \set[\Big]{ u \in \lebesgue^1 \brk{\manifold{M}, \Rset^\nu}\st
 \lim_{r \to 0} \sup_{x \in \manifold{M}}
 \fint_{B_r \brk{x}} \fint_{B_r \brk{x}} \abs{u \brk{y} - u \brk{z}} \dif y \dif z = 0}
 \eqpunct;
\end{equation*}
the space
\(\VMO \brk{\manifold{M}, \Rset^\nu}\) is a Banach space with the \emph{bounded mean oscillation} norm
\[
 \norm{u}_{\BMO}
 = \int_{\manifold{M}} \abs{u}
 + \sup_{0 < r < \rho}
 \sup_{x \in \manifold{M}}
 \fint_{B_r \brk{x}} \fint_{B_r \brk{x}} \abs{u \brk{y} - u \brk{z}} \dif y \dif z\eqpunct,
\]
where \(0 < \rho <\inj \brk{\manifold{M}}\) is fixed and \(\inj \brk{\manifold{M}}\) is the \emph{injectivity radius} of the manifold \(\manifold{M}\);
one then defines
\[
  \VMO \brk{\manifold{M}, \manifold{N}}
 \defeq
 \set{u \in \VMO \brk{\manifold{M}, \Rset^\nu}
 \st u \in \manifold{N} \text{ almost everywhere in }\manifold{M}}
 \eqpunct.
\]

Two maps \(u_0\) and \(u_1 \in \VMO \brk{\manifold{M}, \manifold{N}}\) are defined to be \emph{homotopic in \(\VMO \brk{\manifold{M}, \manifold{N}}\)} whenever there exists a mapping \(H \in \continuous \brk{\intvc{0}{1}, \VMO \brk{\manifold{M}, \manifold{N}}}\) such that \(H \brk{0, \cdot} = u_0\) and \(H \brk{1, \cdot} = u_1\).
Brezis and Nirenberg have proved that the homotopies in \(\VMO \brk{\manifold{M}, \manifold{N}}\) are essentially equivalent to the classical ones in \(\continuous \brk{\manifold{M}, \manifold{N}}\) \cite{Brezis_Nirenberg_1995} (see also \cite{Abbondandolo_1996}).

\begin{proposition}
If \(\manifold{M}\) and \(\manifold{N}\) are compact Riemannian manifolds, then 
\begin{enumerate}[label=(\roman*)]
 \item every \(u \in \VMO \brk{\manifold{M}, \manifold{N}}\) is homotopic in \(\VMO \brk{\manifold{M}, \manifold{N}}\) to some \(v \in \continuous \brk{\manifold{M}, \manifold{N}}\),
 \item for every \(u_0\)  and \(u_1 \in \continuous \brk{\manifold{M}, \manifold{N}}\), the mappings
 \(u_0\) and \(u_1\) are homotopic in \(\VMO \brk{\manifold{M}, \manifold{N}}\) if and only if they are homotopic in \(\continuous \brk{\manifold{M}, \manifold{N}}\).
\end{enumerate}
\end{proposition}

Our main tool is the following sufficient condition for homotopy in \(\VMO \brk{\manifold{M}, \manifold{N}}\), which goes essentially back to Brezis and Nirenberg's work \cite{Brezis_Nirenberg_1995}*{Lem.\ A.19}.

\begin{proposition}
\label{proposition_homotopy_VMO_Rp}
If \(\manifold{M}\) and \(\manifold{N}\) are compact Riemannian manifolds, there exists \(\eta \in \intvo{0}{\infty}\) such that if \(u_0\) and \(u_1 \in \VMO \brk{\manifold{M}, \manifold{N}}\) and if \(\rho \in \intvo{0}{\inj \brk{\manifold{M}}}\) satisfy
\begin{equation}
\label{eq_epif6eeyazeijaiTohpheepa}
\sum_{j = 0}^1
\sup_{\substack{x \in \manifold{M}\\ 0 < r < \rho}}
\fint_{B_r \brk{x}} \fint_{B_r \brk{x}} \abs{u_j \brk{y} - u_j \brk{z}}\dif y \dif z
+ \fint_{B_\rho \brk{x}} \abs{u_0 - u_1} \le \eta\eqpunct,
\end{equation}
then \(u_0\) and \(u_1\) are homotopic in \(\VMO \brk{\manifold{M}, \manifold{N}}\).
\end{proposition}

The proof of \cref{proposition_homotopy_VMO_Rp} is based on Schoen and Uhlenbeck’s averaging argument to estimate the distance of the average \cite{Schoen_Uhlenbeck_1983}.

\begin{proof}[Proof of \cref{proposition_homotopy_VMO_Rp}]
For every \(r \in \intvo{0}{\inj \brk{\manifold{M}}}\) and \(j \in \set{0, 1}\), we define the function
\(u_j^r \colon \manifold{M} \to \Rset^\nu\) for each \(x \in \manifold{M}\) by
\begin{equation*}
  u_j^r \brk{x} \defeq \fint_{B_r \brk{x}} u_j \eqpunct.
\end{equation*}
It follows from our assumption \eqref{eq_epif6eeyazeijaiTohpheepa} that if \(r < \rho\), if \(x \in \manifold{M}\) and if \(j \in \set{0, 1}\) we have
\begin{equation*}
 \dist \brk{u_j^r \brk{x}, \manifold{N}}
 \le \fint_{B_r \brk{x}} \fint_{B_r \brk{x}} \abs{u_j \brk{y} - u_j \brk{z}} \dif y \dif z
 \le \eta \eqpunct{,}
\end{equation*}
and that for every \(x \in \manifold{M}\)
\begin{equation*}
  \abs{u_0^\rho \brk{x} - u_1^\rho \brk{x}}
  \le  \fint_{B_\rho\brk{x}} \abs{u_0 - u_1}
  \leq
  \eta \eqpunct{.}
\end{equation*}
If \(\eta\) is sufficiently small, we get a homotopy by defining \(H \colon \intvc{0}{1} \times \manifold{M} \to \Rset^\nu\) for every \(t \in \intvc{0}{1}\) and \(x \in \manifold{M}\) by
\[
	H\brk{t,x} = \begin{cases}
	    u_0 \brk{x}&\text{if \(t = 0\),}\\
		\Pi_{\manifold{N}} \brk{u_{0}^{3t\rho}\brk{x}} & \text{if \( 0 < t \leq 1/3 \),} \\
		\Pi_{\manifold{N}}\brk{\brk{3t-1}u_{1}^{\rho}\brk{x} + \brk{2-3t}u_{0}^{\rho}\brk{x}} & \text{if \( 1/3 \leq t \leq 2/3 \),} \\
		\Pi_{\manifold{N}} \brk{u_{1}^{3\brk{1 - t}\rho}\brk{x}} & \text{if \( 2/3 \leq t < 1 \),}\\
		u_1 \brk{x}&\text{if \(t = 1\).}\\
	\end{cases}
\]
where \( \Pi_{\manifold{N}} \) is a Lipschitz-continuous retraction of a neighbourhood of \(\manifold{N}\) in \(\Rset^\nu\) onto \( \manifold{N} \).
\resetconstant
\end{proof}

\subsection{Equi-integrability and homotopies}
We first prove that when \(p = \dim \manifold{M}\), the equi-integrable convergence preserves homotopies in \(\VMO \brk{\manifold{M}, \manifold{N}}\).

\begin{proposition}
\label{proposition_homotopy_equiintegrable_Rp}
If \(\manifold{M}\) and \(\manifold{N}\) are compact Riemannian manifolds and if \(p = \dim \manifold{M}\), then there exists \(\eta \in \intvo{0}{\infty}\) such that if \(u_0\), \(u_1 \in \sobolev^{1, p} \brk{\manifold{M}, \manifold{N}}\) and \(\rho \in \intvo{0}{\operatorname{inj}\brk{\manifold{M}}}\) satisfy
\[
 \int_{\manifold{M}} \brk[\Big]{\abs{\Deriv u_0} - \frac{\eta}{\rho}}_+^p
 + \brk[\Big]{\abs{\Deriv u_1} - \frac{\eta}{\rho}}_+^p
 + \frac{\brk{\abs{u_0 - u_1} - \eta}_+}{\rho^p} \le \eta
 \eqpunct{,}
\]
then \(u_0\) and \(u_1\) are homotopic in \(\VMO \brk{\manifold{M}, \manifold{N}}\).
\end{proposition}

The proof of \cref{proposition_homotopy_equiintegrable_Rp} combines \cref{proposition_homotopy_VMO_Rp} with Schoen and Uhlenbeck's estimate on the mean oscillation by the critical Sobolev energy through Poincaré's inequality \cite{Schoen_Uhlenbeck_1983} with a truncation argument.

\begin{proof}[Proof of \cref{proposition_homotopy_equiintegrable_Rp}]
By the Poincaré and Hölder inequalities, if \(2 r < \rho <\inj \brk{\manifold{M}}\), we have for each \(j \in \set{0, 1}\)
\[
\begin{split}
 \fint_{B_r \brk{x}} \fint_{B_r \brk{x}} \abs{u_j \brk{y} - u_j \brk{z}}\dif y \dif z
 &\le \Cl{cst_Tohf1zeeKixoh1aePeim8ZoR} r \fint_{B_r \brk{x}} \abs{\Deriv u_j}\\
 &\le \Cr{cst_Tohf1zeeKixoh1aePeim8ZoR} r \fint_{B_r \brk{x}} \brk[\Big]{\abs{\Deriv u_j} - \frac{\eta}{\rho}}_+ + \Cl{cst_quui2chaetheiNei4aeP0tai} \eta\frac{r}{\rho}\\
 &\le \Cr{cst_Tohf1zeeKixoh1aePeim8ZoR} r\,  \brk[\bigg]{ \fint_{B_r \brk{x}} \brk[\Big]{ \abs{\Deriv u_j} - \frac{\eta}{\rho}}_+^p }^{\frac{1}{p}}
 + \Cr{cst_quui2chaetheiNei4aeP0tai}\eta\frac{r}{\rho}\\
 &\le \Cl{cst_eijeighahNgaeghaiv7ThooV} \brk[\bigg]{\int_{\manifold{M}} \brk[\Big]{\abs{\Deriv u_j} - \frac{\eta}{\rho}}_+^p}^{\frac{1}{p}} + \Cr{cst_quui2chaetheiNei4aeP0tai} \eta
 \eqpunct.
 \end{split}
\]
On the other hand, we have
\[
\begin{split}
 \fint_{B_\rho \brk{x}} \abs{u_0 - u_1}
 &\le \eta + \fint_{B_\rho \brk{x}} \brk{\abs{u_0 - u_1} - \eta}_+ \\
&\le \eta + \C \int_{B_\rho \brk{x}} \frac{\brk{\abs{u_0 - u_1} - \eta}_+}{\rho^p} \eqpunct{.}
\end{split}
\]
We then conclude thanks to \cref{proposition_homotopy_VMO_Rp} provided \(\eta > 0\) is taken  small enough.
\resetconstant
\end{proof}

As an immediate consequence of \cref{proposition_homotopy_equiintegrable_Rp}, we get the following homotopy criterion in critical dimension \(\dim \manifold{M} = p\).

\begin{proposition}
If \(\manifold{M}\) and \(\manifold{N}\) are compact Riemannian manifolds and if \(p = \dim \manifold{M}\),
then there exists \(\eta \in \intvo{0}{\infty}\) such that if \(\brk{u_n}_{n \in \Nset}\) is a sequence in \(\sobolev^{1, p}\brk{\manifold{M}, \manifold{N}}\), if \(u \in \sobolev^{1, p}\brk{\manifold{M}, \manifold{N}}\),
if
\begin{equation}
\label{eq_cohhaNg4xahroich6QueHeov}
 \lim_{n \to \infty} \int_{\manifold{M}} \brk{\abs{u_n - u} - \eta}_+ = 0\eqpunct{,}
\end{equation}
and if 
\begin{equation}
\label{eq_Bo5ahqu1fe6aiCh4anaecahx}
 \lim_{t \to \infty}
 \sup_{n \in \Nset}
 \int_{\manifold{M}} \brk{\abs{\Deriv u_n} - t}_+^p \le \eta \eqpunct{,}
\end{equation}
then, for \(n \in \Nset\) large enough, the maps  \(u\) and \(u_n\) are homotopic in \(\VMO \brk{\manifold{M}, \manifold{N}}\).
\end{proposition}
\begin{proof}
By our assumption \eqref{eq_Bo5ahqu1fe6aiCh4anaecahx}, there exists \(t_* \in \intvo{0}{\infty}\) such that for every \(n \in \Nset\),
\begin{equation}
\label{eq_goo3AthuSeiQuohPiiwu4wei}
 \int_{\manifold{M}} \brk{\abs{\Deriv u_{n}} - t_*}_+^p \le 2 \eta
\end{equation}
and such that
\begin{equation}
\label{eq_eichayu0zoo1yaeG8sheithi}
 \int_{\manifold{M}} \brk{\abs{\Deriv u} - t_*}_+^p \le 2 \eta \eqpunct .
\end{equation}
Taking \(\rho \defeq 5 \eta/t_*\), we have
then by \eqref{eq_goo3AthuSeiQuohPiiwu4wei} and \eqref{eq_eichayu0zoo1yaeG8sheithi}
\[
 \int_{\manifold{M}} \brk[\Big]{\abs{\Deriv u_{n}} - \frac{5\eta}{\rho}}_+^p
 +
  \int_{\manifold{M}} \brk[\Big]{\abs{\Deriv u} - \frac{5\eta}{\rho}}_+^p \le 4 \eta \eqpunct .
\]
In view of \eqref{eq_cohhaNg4xahroich6QueHeov}, we can further assume that for every \(n \in \Nset\),
\[
  \int_{\manifold{M}} \frac{\brk{\abs{u_{n} - u} - \eta}_+}{\rho^p} \le \eta \eqpunct{.}
\]The conclusion then follows from \cref{proposition_homotopy_equiintegrable_Rp}.
\end{proof}

In order to construct lower-dimensional homotopies,
we define for \(p \in \Nset\) and \(m \in \Nset\) satisfying \(p \le m\), the set \(\mathscr{Q}^{m, p}\)of \(p\)-dimensional faces of unit cubes of \(\Rset^m\) with integer coordinates. Given an open set \(\Omega \subseteq \Rset^m\), we also define for \(\varepsilon \in \intvo{0}{\infty}\) and \(\xi \in \Qset^m_\varepsilon \defeq \intvc{0}{\varepsilon}^m\) the set
\begin{equation}
\label{eq_roh8ahga7eishieSh3jie7oo}
K^{\Omega, p}_{\varepsilon, \xi}
 \defeq
 \bigcup_{\substack{\sigma \in  \mathscr{Q}^{m, p}\\ \varepsilon \sigma + \xi \subseteq \Omega}}  \varepsilon \sigma + \xi\eqpunct{.}
\end{equation}

We say that \(\Sigma \subseteq K^{\Omega, p}_{\varepsilon, \xi}\) is a \emph{regular \(p\)-dimensional cubical subcomplex} whenever
\begin{enumerate}[label=(\roman*)]
 \item (homogeneity) every face \(\sigma \subseteq \Sigma\) is contained in a \(p\)-dimensional face of \(\Sigma\),
 \item (local strong connectedness) for every face \(\sigma \subseteq \Sigma\), the graph whose vertices are the  \(p\)-dimensional faces containing \(\sigma\) which are joined by an edge when they have a \(\brk{p - 1}\)-dimensional intersection containing \(\sigma\) is connected.
\end{enumerate}

For locally strongly connected simplicial complexes see \citelist{\cite{Mohar_1988}*{\S 2}\cite{Izmestiev_Joswig_2003}};
this condition can be rephrased by saying that the union of the relative interiors of the \(p\)-dimensional faces containing \(\sigma\) and of the relative interiors of \(\brk{p-1}\)-dimensional faces containing \(\sigma\) is connected; it is automatically satisfied when \(p = 1\).

The following statement estimates the proportion of skeletons of a given scale on which homotopies fail by truncated integrals.

\begin{proposition}
\label{proposition_W1p_good_grid_measure}
For every compact Riemannian manifold \(\manifold{N}\) and every \(p \in \Nset \setminus \set{0}\), there exists a constant \(\eta > 0\) such that if \(\Omega \subseteq \Rset^m\)  with \(m \ge p\) is open and bounded, if \(\varepsilon > \rho > 0\) and if \(u_0\) and \(u_1 \in \sobolev^{1, p}\brk{\Omega, \manifold{N}}\), then
\[
 \eta \, \mathcal{L}^m \brk{
 E_{\varepsilon}}
 \le \tbinom{m}{p}
 \int_{\Omega}
  \brk[\Big]{\abs{\Deriv u_0} - \frac{\eta}{\rho}}_+^p
 + \brk[\Big]{\abs{\Deriv u_1} - \frac{\eta}{\rho}}_+^p
 + \frac{\brk{\abs{u_0 - u_1} - \eta}_+}{\rho^p}\eqpunct{,}
\]
where
\begin{multline*}
 E_{\varepsilon}
 \defeq
 \bigl\{\xi \in \intvc{0}{\varepsilon}^m \st
 u_0 \restr{\Sigma} \text{ and } u_1 \restr{\Sigma}
 \text{ are not homotopic in \(\VMO \brk{\Sigma, \manifold{N}}\)}\\
 \text{for some regular \(p\)-dimensional cubical subcomplex \(\Sigma \subseteq K^{\Omega, p}_{\varepsilon, \xi}\)
 }\bigr\}\eqpunct{.}
\end{multline*}
\end{proposition}

Here and in the sequel \(\mathcal{L}^m\) denotes the \(m\)-dimensional Lebesgue measure on the Euclidean space \(\Rset^m\).

Even though the skeleton \(\Sigma\) is \emph{not a compact Riemannian manifold}, the previous results of this section still hold for a regular \(p\)-dimensional cubical complex with an injectivity radius of the order of \(\varepsilon\) (see for example \cite{White_1988}).

\begin{proof}[Proof of \cref{proposition_W1p_good_grid_measure}]
Noting that \cref{proposition_homotopy_equiintegrable_Rp} remains valid on a regular \(p\)-dimensional cubical complex such as \(\Sigma\),
it follows from \cref{proposition_homotopy_equiintegrable_Rp} that there exists \(\eta > 0\) such that
\[
 E_\varepsilon
 \subseteq
 \set[\bigg]{\xi \in\intvc{0}{\varepsilon}^m \st
   \int_{K^{\Omega, p}_{\varepsilon, \xi}}
  \brk[\Big]{\abs{\Deriv u_0} - \frac{\eta}{\rho}}_+^p
 + \brk[\Big]{\abs{\Deriv u_1} - \frac{\eta}{\rho}}_+^p
 + \frac{\brk{\abs{u_0 - u_1} - \eta}_+}{\rho^p}  > \eta
 }\eqpunct{}.
\]
By the Chebyshev inequality and Fubini’s theorem, we have
\[
\begin{split}
\mathcal{L}^m \brk{
 E_{\varepsilon}}
 &\le \frac{1}{\eta}
 \int_{\intvc{0}{\varepsilon}^m} \int_{K^{\Omega, p}_{\varepsilon, \xi}}
  \brk[\Big]{\abs{\Deriv u_0} - \frac{\eta}{\rho}}_+^p
 + \brk[\Big]{\abs{\Deriv u_1} - \frac{\eta}{\rho}}_+^p
 + \frac{\brk{\abs{u_0 - u_1} - \eta}_+}{\rho^p}\\
 &\le \frac{\binom{m}{p}}{\eta}
 \int_{\Omega}
  \brk[\Big]{\abs{\Deriv u_0} - \frac{\eta}{\rho}}_+^p
 + \brk[\Big]{\abs{\Deriv u_1} - \frac{\eta}{\rho}}_+^p
 + \frac{\brk{\abs{u_0 - u_1} - \eta}_+}{\rho^p}
 \eqpunct{,}
 \end{split}
\]
which gives the conclusion.
\end{proof}
As a consequence of \cref{proposition_W1p_good_grid_measure}, we get the following property of the equi-integrable limit in Sobolev spaces.

\begin{proposition}
\label{corollary_equiintegrable_W1p_skeletons}
Let \(\manifold{N}\) be a compact Riemannian manifold and let \(p \in \set{1, \dotsc, m - 1}\). There exists \(\eta \in \intvo{0}{\infty}\) such that if \(\Omega \subseteq \Rset^m\) with \(m \ge p\) is open and bounded and if \(\brk{u_n}_{n \in \Nset}\) is a sequence in \(\sobolev^{1, p} \brk{\Omega, \manifold{N}}\) satisfying
\[
 \lim_{n \to \infty} \int_{\Omega}
 \brk{\abs{u_n - u} - \eta}_+ = 0
\]
and
\[
 \lim_{t \to \infty}
 \sup_{n \in \Nset}
 \int_{\Omega} \brk{\abs{\Deriv u_n} - t}_+^p = 0\eqpunct{,}
\]
then there exists a subsequence \(\brk{u_{n_k}}_{k \in \Nset}\) such that for every \(\varepsilon \in \intvo{0}{\infty}\), for almost every \(\xi \in \intvc{0}{\varepsilon}^m\), for every \(k \in \Nset\) large enough and for every regular \(p\)-dimensional cubical complex \(\Sigma \subseteq K^{\Omega, p}_{\varepsilon, \xi}\),
\(\smash{u\restr{\Sigma}}\) and \(\smash{u_{n_k} \restr{\Sigma}}\) are homotopic in \(\VMO \brk{\Sigma, \manifold{N}}\).
\end{proposition}
\begin{proof}
We consider a sequence \(\brk{t_k}_{k \in \Nset}\) such that \(t_k \to \infty\) and a subsequence \(\brk{u_{n_k}}_{k \in \Nset}\) of the sequence \(\brk{u_n}_{n \in \Nset}\) such that
\begin{equation}
\label{eq_nae7gik0uethei6oaqu2Dohg}
 \sum_{k \in \Nset}
 \int_{\Omega} \brk{\abs{\Deriv u_{n_k}}- t_k}_+^p
 + \brk{\abs{\Deriv u}- t_k}_+^p + t_k^p \brk{\abs{u_{n_k} - u} - \eta}_+ < \infty \eqpunct{.}
\end{equation}
Given \(\varepsilon \in \intvo{0}{\infty}\), if  \(k \in \Nset\) is large enough so that
\[
 t_k \ge \frac{\eta}{\varepsilon}\eqpunct{,}
\]
it follows from \cref{proposition_W1p_good_grid_measure} that
\begin{equation}
\label{eq_Xi8vu1Ei8Ungey7veeghohdo}
 \mathcal{L}^m \brk{
 E_{\varepsilon}^k}
 \le \frac{\tbinom{m}{p}}{\eta}
 \int_{\Omega}
  \brk{\abs{\Deriv u_{n_k}} - t_k}_+^p
 + \brk{\abs{\Deriv u} - t_k}_+^p
 + t_k^p\frac{\brk{\abs{u_{n_k} - u} - \eta}_+}{\eta^p}\eqpunct{,}
\end{equation}
where
\begin{multline*}
 E_{\varepsilon}^k
 \defeq
 \bigl\{\xi \in \intvc{0}{\varepsilon}^m \st
 u \restr{\Sigma} \text{ and } u_{n_k} \restr{\Sigma}
 \text{ are not homotopic in \(\VMO \brk{\Sigma, \manifold{N}}\)}\\
 \text{for some regular \(p\)-dimensional cubical subcomplex \(\Sigma \subseteq K^{\Omega, p}_{\varepsilon, \xi}\)
 }\bigr\}\eqpunct{.}
\end{multline*}
We have thus by \eqref{eq_nae7gik0uethei6oaqu2Dohg} and \eqref{eq_Xi8vu1Ei8Ungey7veeghohdo}
\[
 \sum_{\substack{k \in \Nset\\ \eta \le t_k \varepsilon}}  \mathcal{L}^m \brk{
 E_{\varepsilon}^k} < \infty\eqpunct{,}
\]
so that
\[
\lim_{\ell \to \infty}
 \mathcal{L}^m \brk[\Big]{\bigcup_{k \ge \ell} E_\varepsilon^k} = 0\eqpunct{,}
\]
and the conclusion follows.
\end{proof}

\subsection{Characterisation of the equi-integrable closure}
We are now in position to prove a local version of \cref{theorem_W1p_equiintegrable_strong}.

\begin{proposition}
\label{proposition_W1p_equiintegrable_strong_local}
If the set \(\Omega \subset \Rset^m\) is open and bounded with a smooth boundary, if \(\manifold{N}\) is a compact Riemannian manifold and if \(p \in \set{1, \dotsc, m - 1}\), then
\begin{equation*}
  \soboleh^{1, p}_{\mathrm{St}} \brk{\Omega, \manifold{N}}
 =
 \soboleh^{1, p}_{\mathrm{Ei}} \brk{\Omega, \manifold{N}}\eqpunct{.}
\end{equation*}
\end{proposition}

\begin{proof}
This is a consequence of \cref{corollary_equiintegrable_W1p_skeletons} and criteria for the strong approximability of Sobolev mappings \citelist{\cite{Bethuel_1991}*{Proof of Thm.\ 3}\cite{Hang_Lin_2003_II}*{Rem.\ 6.1}\cite{Isobe_2005}*{Thm.\ 1.6}\cite{Bousquet_Ponce_VanSchaftingen_2501_18149}}.
\end{proof}

Let us now present several proofs of \cref{theorem_W1p_equiintegrable_strong}.
A first approach is to combine \cref{proposition_W1p_equiintegrable_strong_local} with a localisation argument.

\begin{proof}[Proof of \cref{theorem_W1p_equiintegrable_strong} by localization]
Assume that \(u \in  \smash{\soboleh^{1, p}_{\mathrm{Ei}} \brk{\manifold{M}, \manifold{N}}}\).
Then by \cref{proposition_W1p_equiintegrable_strong_local} and a local chart argument, there are bounded open sets \(\Omega_1, \dotsc, \Omega_r \subseteq \manifold{M}\) with a smooth boundary such that \(\manifold{M} \subseteq \bigcup_{i = 1}^r \Omega_i\) and \(\smash{u\restr{\Omega_i} \in \soboleh^{1, p}_{\mathrm{Ei}} \brk{\Omega_i, \manifold{N}}} = \smash{\soboleh^{1, p}_{\mathrm{St}} \brk{\Omega_i, \manifold{N}}}\).
Assuming that \(\manifold{M}\) is isometrically embedded in \(\Rset^\mu\) and that \(\Pi_{\manifold{M}}\colon \manifold{M} + B_\delta\to \manifold{M}\) is a smooth retraction, since \(u \in \smash{\soboleh^{1, p}_{\mathrm{Bd}} \brk{\manifold{M}, \manifold{N}}}\),
for almost every \(\xi \in \smash{B_\delta}\), we have \(u \compose \smash{\brk{\Pi_\manifold{M} + \xi}\restr{\manifold{M}^{\floor{p - 1}}}} = \smash{V_\xi \restr{\manifold{M}^{\floor{p - 1}}}}\) almost everywhere in \(\smash{\manifold{M}^{\floor{p - 1}}}\), with \(V_\xi \in \continuous \brk{\manifold{M}, \manifold{N}}\) \cite{Hang_Lin_2003_II}*{Thm.\ 7.1}.
By Isobe’s criterion of global strong approximability from local strong approximability \cite{Isobe_2005}*{Thm.\ 1.6}, we have \(u \in \smash{\soboleh^{1, p}_{\mathrm{St}}} \brk{\manifold{M}, \manifold{N}}\).
\end{proof}

Another approach to obtain \cref{theorem_W1p_equiintegrable_strong} is to start from the following straightforward global counterpart of \cref{proposition_W1p_good_grid_measure}.

\begin{proposition}
\label{proposition_W1p_good_skeleton_measure}
Let \(\manifold{M}\) and \(\manifold{N}\) be compact Riemannian manifolds and let \(p \in \set{1, \dotsc,  \dim \manifold{M}- 1}\).
For every \(p\)-dimensional skeleton \(\manifold{M}^p\) of \(\manifold{M}\), every \(\delta \in \intvo{0}{\infty}\) and every smooth retraction \(\Pi_{\manifold{M}} \colon \manifold{M} + B_\delta \subseteq \Rset^\mu \to \manifold{M}\), there exist constants \(\eta > 0\) and \(\rho > 0\) such that if \(u_0, u_1 \in \sobolev^{1, p}\brk{\manifold{M}, \manifold{N}}\), then
\[
 \eta \mathcal{L}^\mu \brk{
 E}
 \le
 \int_{\manifold{M}}
  \brk[\Big]{\abs{\Deriv u_0} - \frac{\eta}{\rho}}_+^p
 + \brk[\Big]{\abs{\Deriv u_1} - \frac{\eta}{\rho}}_+^p
 + \frac{\brk{\abs{u_0 - u_1} - \eta}_+}{\rho^p}\eqpunct{,}
\]
where
\begin{multline*}
 E
 \defeq
 \bigl\{\xi \in B_\delta \st
 u_0\compose \brk{\Pi_{\manifold{M}} + \xi}\restr{\manifold{M}^{p}} \text{ and }
 u_1 \compose \brk{\Pi_{\manifold{M}} + \xi}\restr{\manifold{M}^{p}}\\
 \text{ are not homotopic in \(\VMO \brk{\manifold{M}^p, \manifold{N}}\)}\bigr\}
 \eqpunct{.}
\end{multline*}
\end{proposition}

\begin{proof}[Proof of \cref{theorem_W1p_equiintegrable_strong} through generic triangulation]
If \(u \in \soboleh^{1, p}_{\mathrm{Ei}} \brk{\manifold{M}, \manifold{N}}\), then it follows from \cref{proposition_W1p_good_skeleton_measure}
that for every \(p\)-dimensional skeleton \(\manifold{M}^p\) of \(\manifold{M}\) and every smooth retraction \(\Pi_{\manifold{M}} \colon \manifold{M} + B_\delta \to \manifold{M}\), for almost every \(\xi \in B_\delta\), the map
\(u\compose \brk{\Pi_{\manifold{M}} + \xi}\restr{\manifold{M}^{p}}\) is homotopic in \(\VMO \brk{\manifold{M}^p, \manifold{N}}\) to \(V_\xi\restr{\manifold{M}^p}\) for some \(V_\xi \in \continuous \brk{\manifold{M}, \manifold{N}}\).
We conclude thanks to strong approximability criteria  \citelist{\cite{Bethuel_1991}*{Proof of Thm.\ 3}\cite{Hang_Lin_2003_II}*{Rem.\ 6.1}\cite{Bousquet_Ponce_VanSchaftingen_2501_18149}}.
\end{proof}

A last approach to the proof of \cref{theorem_W1p_equiintegrable_strong} is to rely on the generic screening techniques  developed by Bousquet, Ponce and the author \cite{Bousquet_Ponce_VanSchaftingen_2501_18149}, based on ideas of Fuglede \cite{Fuglede_1957}*{Thm.\ 3 (f)}.

We start from the case \(p = 1\), where we have the following property showing that equi-integrable sequences give homotopic maps on generic Lipschitz images of \(1\)-dimensional complexes.

\begin{proposition}
\label{proposition_Fuglede_W11_Ei}
For every compact Riemannian manifold \(\manifold{N}\), there exists \(\eta \in \intvo{0}{\infty}\) such that if \(\manifold{M}\) is a compact Riemannian manifold and if \(\brk{u_n}_{n \in \Nset}\) is a sequence in \(\sobolev^{1, 1}\brk{\manifold{M}, \manifold{N}}\)
satisfying
\[
 \lim_{n \to \infty} \int_{\manifold{M}} \brk{\abs{u_n - u} - \eta}_+ = 0
\]
and
\[
 \lim_{t \to \infty}
 \sup_{n \in \Nset}
 \int_{\manifold{M}} \brk{\abs{\Deriv u_n} - t}_+ = 0\eqpunct,
\]
then there exists a subsequence \(\brk{u_{n_k}}_{k \in \Nset}\) and a Borel-measurable function \(w \colon \manifold{M} \to \intvc{0}{\infty}\) such that
\begin{equation}
\label{eq_nailichuequu3Pieteoz8ij7}
 \int_{\manifold{M}} w < \infty
\end{equation}
and such that if \(\Sigma\) is a finite one-dimensional homogeneous simplicial complex,
if \(\gamma \colon \Sigma \to \manifold{M}\) is Lipschitz-continuous and satisfies
\begin{equation}
\label{eq_shieraac8Voh9Dai4iegeiSh}
 \int_{\Sigma}w \compose \gamma < \infty
 \eqpunct,
\end{equation}
then for \(k \in \Nset\) large enough, the maps
\(u \compose \gamma\) and \(u_{n_k} \compose \gamma\) are homotopic in \(\continuous \brk{\Sigma, \manifold{N}}\).
\end{proposition}

Here and in the sequel, each simplex of the one-dimensional simplicial complex \(\Sigma\) is endowed with its Euclidean structure as a unit interval.

\begin{proposition}
\label{proposition_W11_uniform_path}
If \(\manifold{M}\) is a compact Riemannian manifold, if \(\nu \in \Nset\) and if \(u \in \sobolev^{1, 1} \brk{\manifold{M}, \Rset^\nu}\),
then there exists a function \(w \colon \manifold{M} \to \intvc{0}{\infty}\) such that if \(\Sigma\) is a finite \(1\)-dimensional simplicial complex, if \(\gamma \colon \Sigma \to \manifold{M}\) is Lipschitz-continuous, if
\begin{equation}
\label{eq_Rae7xiighoolao4yo0aeS2qu}
 \int_{\Sigma} w \compose \gamma < \infty\eqpunct{,}
\end{equation}
if \(\eta \in \intvr{0}{\infty}\) and if \(\rho \in \intvo{0}{1}\), then
\[
\norm{u \compose \gamma}_{\lebesgue^\infty \brk{\Sigma, \Rset^\nu}}
\le 2\eta +
   \int_{\Sigma}
  \brk[\Big]{\seminorm{\gamma}_{\mathrm{Lip}} \brk{\abs{\Deriv u}\compose \gamma} - \frac{\eta}{\rho}}_+ + \int_{\Sigma} \frac{\brk{\abs{u} - \eta}_+}{\rho}
  \eqpunct.
\]
\end{proposition}

Here and in the sequel \(\seminorm{\gamma}_{\mathrm{Lip}}\) denotes the \emph{Lipschitz constant} of \(\gamma\):
\[
 \seminorm{\gamma}_{\mathrm{Lip}}
 \defeq \sup_{x, y \in \Gamma} \frac{d \brk{\gamma \brk{x}, \gamma \brk{y}}} {d \brk{x, y}}\eqpunct{.}
\]

The proof of \cref{proposition_W11_uniform_path} relies on the following one-dimensional truncated Sobolev embedding.

\begin{lemma}
\label{lemma_truncated_1d}
If \(u \in \sobolev^{1, 1} \brk{\intvo{0}{1}, \Rset^\nu}\), then for every \(\rho \in \intvo{0}{1}\), every \(\eta \in \intvr{0}{\infty}\) and for almost every \(t \in \intvo{0}{1}\),
\begin{equation}
\label{eq_Sie7yeungaiKu6miereivi5a}
 \abs{u\brk{t}}
 \le 2 \eta + \int_0^1 \brk[\Big]{\abs{u'} - \frac{\eta}{\rho}}_+ + \frac{\brk{\abs{u} - \eta}_+}{\rho}\eqpunct.
\end{equation}
\end{lemma}
\begin{proof}
Assuming without loss of generality that \(t \in \intvo{a}{a + \rho} \subseteq \intvo{0}{1}\) for some \(a \in \intvc{0}{1}\), we have
\[
\begin{split}
 u \brk{t}
 &= \frac{1}{\rho}\int_a^{a + \rho} \int_{\tau}^{t} u'\brk{\sigma} \dif \sigma + u \brk{\tau} \dif \tau\\
&= \int_a^{t} \frac{\tau - a}{\rho} u'\brk{\tau}\dif \tau -
\int_{t}^{a + \rho} \frac{a + \rho - \tau}{\rho} u'\brk{\tau}\dif \tau
+
\int_{a}^{a + \rho} \frac{u \brk{ \tau}}{\rho} \dif \tau
 \eqpunct{,}
\end{split}
\]
and thus
\[
 \abs{u \brk{t}}
 \le \int_{a}^{a + \rho} \abs{u'} + \frac{\abs{u}}{\rho}\\
 \le2 \eta +  \int_0^1 \brk[\Big]{\abs{u'} - \frac{\eta}{\rho}}_+ + \frac{\brk{\abs{u} - \eta}_+}{\rho}
 \eqpunct{,}
\]
from which the conclusion \eqref{eq_Sie7yeungaiKu6miereivi5a} follows.
\end{proof}

We now proceed to the proof of \cref{proposition_W11_uniform_path}.

\begin{proof}[Proof of \cref{proposition_W11_uniform_path}]
Following
\citelist{\cite{Fuglede_1957}*{Thm.\ 3 (f)}\cite{Bousquet_Ponce_VanSchaftingen_2501_18149}},
let \(\brk{u_n}_{n \in \Nset}\) be a sequence in \(\smooth^{\infty}\brk{\manifold{M}, \Rset^\nu}\) such that
\[
 \sum_{n \in \Nset} \int_{\manifold{M}} \abs{\Deriv u_n - \Deriv u} + \abs{u_n - u}<\infty
\]
and let
\[
 w \defeq \sum_{n \in \Nset} \abs{\Deriv u_n - \Deriv u} + \abs{u_n - u}\eqpunct{.}
\]
Given a Lipschitz-continuous mapping \(\gamma \colon \Sigma \to \manifold{M}\) satisfying \eqref{eq_Rae7xiighoolao4yo0aeS2qu}, it then follows that \(u \compose \gamma \in \sobolev^{1, 1}\brk{\Sigma, \Rset^\nu}\) and
\(\brk{u\compose \gamma}' = \brk{\Deriv u \compose \gamma }\gamma'\) almost everywhere on \(\Sigma\).
The conclusion then follows from \cref{lemma_truncated_1d}.
\end{proof}

We are now in position to prove \cref{proposition_Fuglede_W11_Ei}.

\begin{proof}[Proof of \cref{proposition_Fuglede_W11_Ei}]
By assumption, there exists a sequence \(\brk{t_k}_{k \in \Nset}\) and an increasing sequence \(\brk{n_k}_{k\in \Nset}\) in \(\Nset\)  such that \(t_k \to \infty\) and
\begin{equation}
\label{eq_eShae2hipohBa9CheiBeu2Ri}
 \sum_{k \in \Nset}
  \int_{\manifold{M}}
 \brk{ \abs{\Deriv u_{n_k}} - t_k}_+ + \brk{\abs{\Deriv u } - t_k}_+
 + t_k \brk{\abs{u_{n_k} - u} - \eta}_+
 < \infty
 \eqpunct.
\end{equation}
We then define the function \(w \colon \manifold{M} \to \intvc{0}{\infty}\) by
\begin{equation}
\label{eq_Uyoopoo2Ohwie5Shuj6aPhie}
 w \defeq \sum_{k \in \Nset}
 \brk{ \abs{\Deriv u_{n_k}} - t_k}_+ + \brk{\abs{\Deriv u } - t_k}_+
 + t_k \brk{\abs{u_{n_k} - u} - \eta}_+
 \eqpunct.
\end{equation}
The condition \eqref{eq_nailichuequu3Pieteoz8ij7} then follows from \eqref{eq_eShae2hipohBa9CheiBeu2Ri} and \eqref{eq_Uyoopoo2Ohwie5Shuj6aPhie}.
By increasing the value of \(w\) if necessary, we can assume that the conclusion of \cref{proposition_W11_uniform_path} applies with \(w\) to the maps \(u_{n_k}- u\) respectively.

Given a Lipschitz-continuous mapping \(\gamma \colon \Sigma \to \manifold{M}\) such that the condition \eqref{eq_shieraac8Voh9Dai4iegeiSh} holds, by \cref{proposition_W11_uniform_path},  we have for every \(k \in \Nset\), \(\eta \in \intvo{0}{\infty}\) and \(\rho \in \intvo{0}{1}\),
\[
\begin{split}
&\norm{u \compose \gamma - u_{n_k} \compose \gamma}_{\lebesgue^\infty\brk{\Sigma, \Rset^\nu}}\\
&\quad \le 2 \eta +
   \int_{\Sigma}
  \brk[\Big]{\seminorm{\gamma}_{\mathrm{Lip}} \brk{\abs{\Deriv u - \Deriv u_{n_k}}\compose \gamma} - \frac{\eta}{\rho}}_+ + \frac{\brk{\abs{u_{n_k} - u} - \eta}_+}{\rho}\\
&\quad \le 2 \eta +
   \int_{\Sigma}
  \brk[\Big]{\seminorm{\gamma}_{\mathrm{Lip}} \brk{\abs{\Deriv u_{n_k}}\compose \gamma} - \frac{\eta}{2\rho}}_+ \!
  + \brk[\Big]{\seminorm{\gamma}_{\mathrm{Lip}} \brk{\abs{\Deriv u}\compose \gamma} - \frac{\eta}{2\rho}}_+  \!
  + \frac{\brk{\abs{u_{n_k} - u} - \eta}_+}{\rho}
  \eqpunct.
\end{split}
\]
Taking
\[
  \rho_k \defeq \frac{\eta}{2 t_k \seminorm{\gamma}_{\mathrm{Lip}}}
  \eqpunct{,}
\]
we have then, if \(k \in \Nset\) is large enough so that  \(\rho_k \le 1\),
\begin{multline*}
 \norm{u \compose \gamma - u_{n_k} \compose \gamma}_{\lebesgue^\infty\brk{\Sigma, \Rset^\nu}}\\
 \le 2\eta + \seminorm{\gamma}_{\mathrm{Lip}} \int_{\Sigma} \brk{\abs{\Deriv u_{n_k}}\compose \gamma - t_k}_+ + \brk{\abs{\Deriv u}\compose \gamma - t_k}_+
 + \frac{2 t_k  \brk{\abs{u_{n_k} - u} \compose \gamma - \eta}_+}{ \eta} \eqpunct{.}
\end{multline*}
It follows from \eqref{eq_shieraac8Voh9Dai4iegeiSh} and \eqref{eq_Uyoopoo2Ohwie5Shuj6aPhie} that
\[
  \lim_{k \to \infty}
  \int_{\Sigma} \brk{\abs{\Deriv u_{n_k}}\compose \gamma - t_k}_+ + \brk{\abs{\Deriv u}\compose \gamma - t_k}_+
 + t_k \brk{\abs{u_{n_k} - u} \compose \gamma - \eta}_+ = 0
\]
and thus
\[
 \limsup_{k \to \infty} \, \norm{u \compose \gamma - u_{n_k} \compose \gamma}_{\lebesgue^\infty\brk{\Sigma, \Rset^\nu}} \le 2 \eta \eqpunct.
\]
If \(\eta\) is small enough, this implies that the mappings \(u\) and \(u_{n_k}\) are homotopic in \(\continuous \brk{\Sigma, \manifold{N}}\) for \(k\) large enough.
\end{proof}

The counterpart of \cref{proposition_Fuglede_W11_Ei} for \(p > 1\) reads as:

\begin{proposition}
\label{proposition_Fuglede_W1p_Ei}
For every compact Riemannian manifold \(\manifold{N}\), there exists \(\eta \in \intvo{0}{\infty}\) such that if \(\manifold{M}\) is a compact Riemannian manifold and if \(\brk{u_n}_{n \in \Nset}\) is a sequence in \(\sobolev^{1, p}\brk{\manifold{M}, \manifold{N}}\)
satisfying
\[
 \lim_{n \to \infty} \int_{\manifold{M}} \brk{\abs{u_n - u} - \eta}_+ = 0
\]
and
\[
 \lim_{t \to \infty}
 \sup_{n \in \Nset}
 \int_{\manifold{M}} \brk{\abs{\Deriv u_n} - t}_+^p = 0\eqpunct,
\]
then there exists a subsequence \(\brk{u_{n_k}}_{k \in \Nset}\) and a Borel-measurable function \(w \colon \manifold{M} \to \intvc{0}{\infty}\) such that
\begin{equation}
\label{eq_zohGhohzu1ahthee3aiquiva}
 \int_{\manifold{M}} w < \infty
\end{equation}
and such that if \(\Sigma\) is a finite \(p\)-dimensional simplicial complex,
if \(\gamma \colon \Sigma \to \manifold{M}\) is Lipschitz-continuous and satisfies
\begin{equation}
\label{eq_keequeishaXi9zai4soodah8}
 \int_{\Sigma}w \compose \gamma < \infty
 \eqpunct,
\end{equation}
then for \(k \in \Nset\) large enough, the mappings
\(u \compose \gamma\) and \(u_{n_k} \compose \gamma\) are homotopic in \(\VMO \brk{\Sigma, \manifold{N}}\).
\end{proposition}

Here and in the sequel, each simplex of the \(p\)-dimensional simplicial complex \(\Sigma\) is endowed with its Euclidean structure as a regular complex, spanned by \(e_0/\sqrt{2}, \dotsc, e_p/\sqrt{2}\), where \(e_0, \dotsc, e_p\) form an orthonormal basis of the Euclidean space \(\Rset^{p + 1}\).

\begin{proposition}
\label{proposition_VMO_maximal_path}
If \(\manifold{M}\) is a compact Riemannian manifold, there exists a constant \(C \in \intvo{0}{\infty}\) such that if \(u \in \sobolev^{1, p} \brk{\manifold{M}, \Rset^\nu}\), if \(\Sigma\) is a finite \(p\)-dimensional homogeneous simplicial complex, if the mapping  \(\gamma \colon \Sigma \to \manifold{M}\) is Lipschitz-continuous and if for almost every \(t \in \Sigma\)
\begin{equation}
\label{eq_aequech9Aingie0eu1aap3Mo}
 \liminf_{\delta \to 0} \fint_{B_\delta \brk{\gamma \brk{t}}} \abs{u \brk{x} - u \brk{\gamma\brk{t}}}\dif x = 0
 \eqpunct,
\end{equation}
then one has for every \(\eta \in \intvr{0}{\infty}\) and every \(r \in \intvo{0}{1}\)
\begin{equation}
\label{eq_ka3ePeodeedoX3oowoo8Huyu}
  \fint_{B_r \brk{a}} \fint_{B_r \brk{a}} \abs{u \brk{\gamma \brk{t}}- u \brk{\gamma\brk{s}}}^p\dif t \dif s
  \le C
   \brk[\bigg]{
\eta^p +
  \int_{\Sigma}
  \brk[\Big]{\mathfrak{M} \brk[\Big]{ \seminorm{\gamma}_{\mathrm{Lip}} \abs{\Deriv u}\compose \gamma - \frac{\eta}{r}}_+}^p}
  \eqpunct.
\end{equation}
\end{proposition}

Here \(\mathfrak{M} f \colon \manifold{M }\to \intvc{0}{+\infty}\) denotes the Hardy--Littlewood maximal function of the function \(f \colon \manifold{M} \to \Rset\) defined for each \(x \in \manifold{M}\) by
\[
\mathfrak{M} f \brk{x}
= \sup_{r > 0} \fint_{B_r \brk{x}} f .
\]

In order to prove \cref{proposition_VMO_maximal_path}, we introduce an additional truncation argument in the proof of the corresponding property for \(\eta = 0\) \cite{Bousquet_Ponce_VanSchaftingen_2501_18149}.

\begin{proof}[Proof of \cref{proposition_VMO_maximal_path}]
By the Lusin--Lipschitz formula \citelist{\cite{Liu_1977}*{Lem.\ 2}\cite{Acerbi_Fusco_1984}*{Lem.\ I.11}\cite{Bojarski_1990}\cite{Hajlasz_1996}*{p. 404}\cite{Jabin_2010}*{(3.3)}}, if
\[
 \liminf_{\delta \to 0} \fint_{B_\delta \brk{y}} \abs{u \brk{x} - u \brk{y}}\dif x =
 \liminf_{\delta \to 0} \fint_{B_\delta \brk{z}} \abs{u \brk{x} - u \brk{z}} \dif x =
 0
 \eqpunct,
\]
then we have
\[
 \abs{u \brk{y} - u \brk{z}}
 \le \Cl{cst_uetaemahZ2fee6wah5Ti8xee}\, d\brk{y, z}\,
 \brk{\mathfrak{M} \abs{\Deriv u} \brk{y} + \mathfrak{M} \abs{\Deriv u} \brk{z}}
 \eqpunct.
\]
For almost every \(s, t \in \Sigma\) it follows from our assumption \eqref{eq_aequech9Aingie0eu1aap3Mo} that
\[
 \abs{u \brk{\gamma \brk{t}} - u \brk{\gamma{\brk{s}}}}
 \le \Cr{cst_uetaemahZ2fee6wah5Ti8xee}\, d\brk{\gamma \brk{s}, \gamma\brk{t}}\,
 \brk{\mathfrak{M} \abs{\Deriv u} \brk{\gamma \brk{t}} + \mathfrak{M} \abs{\Deriv u} \brk{\gamma \brk{s}}}
 \eqpunct,
\]
and therefore, by integration and by symmetry,
\begin{equation}
\label{eq_oH1ooLohmoo5phei1oeboh2F}
   \fint_{B_r \brk{a}} \fint_{B_r \brk{a}} \abs{u \brk{\gamma \brk{t}} - u\brk{\gamma\brk{s}}}^p\dif t \dif s
   \le \C  \int_{B_r \brk{a}}
   \seminorm{\gamma}_{\mathrm{Lip}}^p \brk{\mathfrak{M} \abs{\Deriv u} \brk{\gamma \brk{t}}}^p \dif t
   \eqpunct.
\end{equation}
We next note that
\[
  \seminorm{\gamma}_{\mathrm{Lip}}  \abs{\Deriv u}
 \le \brk[\Big]{\seminorm{\gamma}_{\mathrm{Lip}}  \abs{\Deriv u} - \frac{\eta}{r}}_+ + \frac{\eta}{r}
 \eqpunct,
\]
so that
\[
   \mathfrak{M} \brk{\seminorm{\gamma}_{\mathrm{Lip}}  \abs{\Deriv u}}
  \le \mathfrak{M} \brk[\Big]{\seminorm{\gamma}_{\mathrm{Lip}}  \abs{\Deriv u} - \frac{\eta}{r}}_+ + \frac{\eta}{r}
\]
and finally
\begin{equation}
\label{eq_thei8Bie0eithae3hohSoy1i}
 \brk{\seminorm{\gamma}_{\mathrm{Lip}}\mathfrak{M} \abs{\Deriv u} \brk{\gamma \brk{t}}}^p
 \le 2^{p -1}\brk[\Big]{\mathfrak{M} \brk[\Big]{\seminorm{\gamma}_{\mathrm{Lip}}  \abs{\Deriv u} - \frac{\eta}{r}}_+\brk{\gamma \brk{t}}}^p + \frac{2^{p - 1}\eta^p}{r^p}
 \eqpunct.
\end{equation}
The conclusion \eqref{eq_ka3ePeodeedoX3oowoo8Huyu} follows from \eqref{eq_oH1ooLohmoo5phei1oeboh2F} and \eqref{eq_thei8Bie0eithae3hohSoy1i} since \(\dim \Sigma = p\).
\resetconstant
\end{proof}

We now deduce \cref{proposition_Fuglede_W1p_Ei} from \cref{proposition_VMO_maximal_path}.

\begin{proof}[Proof of \cref{proposition_Fuglede_W1p_Ei}]
We fix a sequence \(\brk{t_k}_{k \in \Nset}\) such that \(t_k \to \infty\) and an increasing sequence \(\brk{n_k}_{k \in \Nset}\) in \(\Nset\) such that
\[
 \sum_{k \in \Nset}
 \int_{\manifold{M}} \brk{\abs{\Deriv u_{n_k}} - t_k}_+^p + \brk{\abs{\Deriv u} - t_k}_+^p
 + t_k^p \brk{\abs{u_{n_k} - u} - \eta}_+ < \infty
 \eqpunct.
\]
Next, we define the function \(w \colon \manifold{M} \to \intvc{0}{\infty}\) by
\[
 w  \defeq \sum_{k \in \Nset}
 \brk{\mathfrak{M} \brk{\abs{\Deriv u_{n_k}} - t_k}_+}^p + \brk{\mathfrak{M} \brk{\abs{\Deriv u} - t_k}_+}^p
 + t_k^p \brk{\abs{u_{n_k} - u} - \eta}_+
\]
on the set of Lebesgue points of every \(u_{n_k}\) and of \(u\) and by \(w \defeq \infty\) elsewhere.

By Lebesgue’s differentiation theorem and the classical maximal function theorem, since \(p > 1\), (see for example \citelist{\cite{Stein_1970}*{Thm.\ I.1}\cite{Duoandikoetxea_2001}*{Thm.\ 2.16}}), we have
\[
\begin{split}
 \int_{\manifold{M}} w
 &\le \sum_{k \in \Nset}
  \int_{\manifold{M}}\brk{\mathfrak{M} \brk{\abs{\Deriv u_{n_k}} - t_k}_+}^p
  +  \int_{\manifold{M}} \brk{\mathfrak{M} \brk{\abs{\Deriv u} - t_k}_+}^p
 + t_k^p  \brk{\abs{u_{n_k} - u} - \eta}_+\\
 &\le \C \sum_{k \in \Nset} \int_{\manifold{M}} \brk{\abs{\Deriv u_{n_k}} - t_k}_+^p + \brk{\abs{\Deriv u} - t_k}_+^p
 + \sum_{k \in \Nset} \int_{\manifold{M}} t_k^p \brk{\abs{u_{n_k} - u} - \eta}_+ < \infty\eqpunct{,}
\end{split}
\]
so that the integrability condition  \eqref{eq_zohGhohzu1ahthee3aiquiva} holds.

Assuming now that \(\Sigma\) is a finite \(p\)-dimensional simplicial complex such that \eqref{eq_keequeishaXi9zai4soodah8} holds, taking
\[
 \rho_k \defeq \frac{\eta}{t_k \seminorm{\gamma}_{\mathrm{Lip}}}
 \eqpunct,
\]
we have then by \cref{proposition_VMO_maximal_path}, for \(r < \rho_k < 1\),
\[
 \frac{\eta}{r} \ge t_k \seminorm{\gamma}_{\mathrm{Lip}}\eqpunct,
\]
and thus
\begin{gather*}
\begin{split}
  \fint_{B_r \brk{a}} \fint_{B_r \brk{a}} \abs{u_{n_k} \brk{\gamma \brk{t}} - u_{n_k}\brk{\gamma\brk{s}}}^p\dif t \dif s
  &\le \Cl{cst_iRoh8eW8ka4iechoh6eZai9f} \brk[\bigg]{
\eta^p +
  \int_{\Sigma} \brk[\Big]{\mathfrak{M} \brk[\Big]{ \seminorm{\gamma}_{\mathrm{Lip}} \abs{\Deriv u_{n_k}}\compose \gamma - \frac{\eta}{r}}_+}^p}\\
  &\le \Cr{cst_iRoh8eW8ka4iechoh6eZai9f} \brk[\bigg]{
\eta^p +
\seminorm{\gamma}_{\mathrm{Lip}}^p
\hspace{-.3em}
  \int_{\Sigma}
  \brk[\big]{\mathfrak{M} \brk{ \abs{\Deriv u_{n_k}}\compose \gamma - t_k}_+}^p}\eqpunct,\\
\end{split}\\
\begin{split}
  \fint_{B_r \brk{a}} \fint_{B_r \brk{a}} \abs{u \brk{\gamma \brk{t}} - u\brk{\gamma\brk{s}}}^p\dif t \dif s
  &\le \Cr{cst_iRoh8eW8ka4iechoh6eZai9f} \brk[\bigg]{
\eta^p +
  \int_{\Sigma} \brk[\Big]{\mathfrak{M} \brk[\Big]{ \seminorm{\gamma}_{\mathrm{Lip}} \abs{\Deriv u}\compose \gamma - \frac{\eta}{r}}_+}^p}\\
  &\le \Cr{cst_iRoh8eW8ka4iechoh6eZai9f} \brk[\bigg]{
\eta^p + \seminorm{\gamma}_{\mathrm{Lip}}^p
  \hspace{-.3em}
  \int_{\Sigma}
  \brk[\big]{\mathfrak{M} \brk{\abs{\Deriv u}\compose \gamma - t_k}_+}^p}\\
\end{split}
\intertext{and}
\int_{\Sigma} \frac{\brk{\abs{u_{n_k} - u}- \eta}_+}{\rho_k^p}
\le \frac{\seminorm{\gamma}_{\mathrm{Lip}}^p}{\eta^p} t_{k}^p
\int_{\Sigma} \brk{\abs{u_{n_k} - u}- \eta}_+\eqpunct.
\end{gather*}
By \eqref{eq_keequeishaXi9zai4soodah8}, the right-hand sides of the previous inequalities tend to \(0\) as \(k \to \infty\).
The conclusion then follows from \cref{proposition_homotopy_VMO_Rp}, provided \(\eta \in \intvo{0}{\infty}\) was taken small enough.
\resetconstant
\end{proof}

\begin{proof}[Proof of \cref{theorem_W1p_equiintegrable_strong} through generic screening]
Let \(u \in \soboleh^{1, p}_{\mathrm{Ei}}\brk{\manifold{M}, \manifold{N}}\) and let the Borel-measurable function \(w \colon \manifold{M}\to \intvc{0}{\infty}\) be given by \cref{proposition_Fuglede_W1p_Ei}.
For every \(\ell \in \Nset\) satisfying \(\ell \ge p\), every
\(\ell\)-dimensional homogeneous simplicial complex
\(\Sigma\) and every Lipschitz-continuous map \(\sigma \colon \Sigma \to \manifold{M}\) such that
\[
 \int_{\Sigma^p} w \compose \sigma \restr{\Sigma^p} <\infty\eqpunct{,}
\]
where \(\Sigma^p\) is the \(p\)-dimensional subskeleton of \(\Sigma\),
the map \(u \compose \sigma\restr{\Sigma^p}  \) is homotopic in \(\mathrm{VMO}\brk{\Sigma^p, \manifold{N}}\) to \(u_{n_k} \compose \sigma\restr{\Sigma^p}  = \brk{u_{n_k} \compose \sigma}\restr{\Sigma^p}\), where \(u_{n_k} \compose \sigma \in \continuous \brk{\Sigma, \manifold{N}}\).
Thanks to the strong approximation criterion through generic screening \cite{Bousquet_Ponce_VanSchaftingen_2501_18149}, we conclude that \(u \in \soboleh^{1, p}_{\mathrm{St}}\brk{\manifold{M}, \manifold{N}}\).
\end{proof}

\section{Higher-order Sobolev spaces}
\label{section_higher_order}
\subsection{Definitions and statements}
As counterparts of \eqref{eq_Roejiejai2zah4iek0auwae0}, \eqref{eq_cho3eibic6eo6ieSh9iPhieh} and \eqref{eq_Thaedah9eeyohtheesho4aic}, we define for \(k \in \Nset\setminus \set{0}\) and \(p \in \intvr{1}{\infty}\),
the \emph{higher-order Sobolev space}
\[
 \sobolev^{k, p}\brk{\manifold{M}, \manifold{N}}
 = \set{u \in \sobolev^{k, p}\brk{\manifold{M}, \Rset^\nu} \st  u \in \manifold{N} \text{ almost everywhere in \(\manifold{M}\)} }\eqpunct{,}
\]
the associated \emph{strong closure of smooth maps}
\begin{multline*}
 \soboleh^{k, p}_{\mathrm{St}} \brk{\manifold{M}, \manifold{N}}
 \defeq \Bigl\{u \in \sobolev^{k, p} \brk{\manifold{M}, \manifold{N}}
 \st \lim_{n \to \infty} \sum_{j = 0}^k  \int_{\manifold{M}} \abs{\Deriv^j u_n - \Deriv^j u}^p = 0\\  \text{ and } u_n \in \smooth^\infty \brk{\manifold{M}, \manifold{N}}\Bigr\}
\end{multline*}
and its \emph{equi-integrable sequential closure}
\begin{multline*}
 \soboleh^{k, p}_{\mathrm{Ei}} \brk{\manifold{M}, \manifold{N}}
 \defeq \Bigl\{u \in \sobolev^{k, p}\brk{\manifold{M}, \manifold{N}}
 \st \lim_{n \to \infty} \int_{\manifold{M}} \abs{u_n - u}^p = 0\eqpunct,\\ \lim_{t \to \infty} \sup_{n \in \Nset} \int_{\manifold{M}} \brk{\abs{\Deriv^k u_n} - t}_+^p = 0 \text{ and } u_n \in \smooth^\infty \brk{\manifold{M}, \manifold{N}}\Bigr\}
 \eqpunct{.}
\end{multline*}
As previously, we have the immediate inclusion
\begin{equation}
\label{eq_yu7Shei1saingoojeic0chie}
  \soboleh^{k, p}_{\mathrm{St}} \brk{\manifold{M}, \manifold{N}}
  \subseteq \soboleh^{k, p}_{\mathrm{Ei}} \brk{\manifold{M}, \manifold{N}}
  \eqpunct{.}
\end{equation}

The result of \cref{theorem_W1p_equiintegrable_strong} extends to higher-order Sobolev spaces.

\begin{theorem}
\label{theorem_Wkp_equiintegrable_strong}
If \(\manifold{M}\) and \(\manifold{N}\) are compact Riemannian manifolds, if \(k \in \Nset \setminus \set{0}\) and if \(p \in \intvr{1}{\infty}\), then
\begin{equation}
\label{eq_saah5uCie3xee0fu1eofo1ue}
  \soboleh^{k, p}_{\mathrm{St}} \brk{\manifold{M}, \manifold{N}}
 =
 \soboleh^{k, p}_{\mathrm{Ei}} \brk{\manifold{M}, \manifold{N}}\eqpunct{.}
\end{equation}
\end{theorem}

Again \cref{theorem_Wkp_equiintegrable_strong} is only relevant when \(kp \in \Nset \setminus \set{0}\) and \(1 \le kp <\dim \manifold{M}\).
In view of \cref{theorem_W1p_equiintegrable_strong}, we only need to prove it for \(k \ge 2\).

\subsection{The case \texorpdfstring{\(p > 1\)}{p>1}}

When \(p > 1\), the proof of \cref{theorem_Wkp_equiintegrable_strong} relies on a relationship between \(\sobolev^{k, p}\)-equi-integrability and \(\sobolev^{1, kp}\)-equi-integrability that refines the classical Gagliardo-Nirenberg interpolation inequality \citelist{\cite{Nirenberg_1959}\cite{Gagliardo_1958}\cite{Gagliardo_1959}} (see also \citelist{\cite{Leoni_2017}*{Thm.\ 12.85}\cite{Brezis_Mironescu_2018}}): for every \(u \in \sobolev^{k,p}\brk{\Rset^m, \Rset^\nu}\cap \lebesgue^\infty \brk{\Rset^m, \Rset^\nu}\)
\begin{equation}
\label{eq_ohgohPeed5fieRiere6Pohxo}
\int_{\Rset^m} \abs{\Deriv u}^{kp}
\le C
\norm{u}_{\lebesgue^\infty\brk{\Rset^m, \Rset^\nu}}^{\brk{k-1}p}
 \int_{\Rset^m} \abs{\Deriv^k u}^{p}\eqpunct{.}
\end{equation}

\begin{proposition}
\label{proposition_Wkp_trunc_W1kp_trunc}
Given \(m, k \in \Nset \setminus \set{0}\) and \(p \in \intvr{1}{\infty}\), there exist constants \(C \in \intvo{0}{\infty}\) and \(\kappa \in \intvo{0}{\infty}\) such that if \(\Omega \subseteq \Rset^m\) is open, if \(\nu \in \Nset \setminus \set{0}\), if \(u \in \sobolev^{k, 1}_{\mathrm{loc}} \brk{\Omega, \Rset^\nu}\),
if
\begin{equation}
\label{eq_IeZohY8amulieB2eemiiB9ph}
  \norm{u}_{\lebesgue^\infty \brk{\Omega, \Rset^\nu}}^{k} \le \kappa t^k \rho^k
  \eqpunct{,}
\end{equation}
and if \(K \subseteq \Rset^m\) is measurable and satisfies
\begin{equation}
 K + B_\rho \subseteq \Omega\eqpunct{,}
\end{equation}
then
\begin{equation}
\label{eq_aeVoochae8deraegh7iekees}
\int_{K}  \brk{\abs{\Deriv u} - t}_+^{kp}
\le
C \int_{\Omega} \brk[\big]{\norm{u}_{\lebesgue^\infty\brk{\Omega, \Rset^\nu}}^{k-1} \abs{\Deriv^k u}
  - \kappa t^k}_+^{p}\eqpunct.
\end{equation}
\end{proposition}

\Cref{proposition_Wkp_trunc_W1kp_trunc} is a counterpart for the Gagliardo-Nirenberg interpolation inequality of the preservation of equi-integrability by the endpoint Sobolev embedding \cite{DeLellis_Focardi_Spadaro_2011}*{Lem.\ A.3}.

The essential ingredient of the proof of \cref{proposition_Wkp_trunc_W1kp_trunc} is the following pointwise interpolation inequality, which can also be used to prove the classical Gagliardo-Nirenberg interpolation inequality \eqref{eq_ohgohPeed5fieRiere6Pohxo}.
This estimate is an immediate consequence of a pointwise interpolation inequality of Maz'ya and Shaposhnikova \cite{Mazya_Shaposhnikova_1999}*{Thm.\ 1 and Rem.\ 3}.

\begin{proposition}
\label{proposition_pointwise_interpolation}
Given \(m, k \in \Nset \setminus \set{0}\), there exists a constant \(C \in \intvo{0}{\infty}\) such that if \(\Omega \subseteq \Rset^m\) is open, if \(\nu \in \Nset \setminus \set{0}\) and if \(u \in \sobolev^{k, 1} \brk{\Omega, \Rset^\nu}\), then for almost every \(x \in \Omega\) such that \(B_\rho \brk{x}\subseteq \Omega\), one has
\begin{equation}
\label{eq_thiovohno4xe6lae5Oez4iSh}
\abs{\Deriv u\brk{x}}
 \le C \brk[\Big]{\norm{u}_{\lebesgue^\infty\brk{\Omega}}^{1 - 1/k} \brk{\mathfrak{M} \abs{\Deriv^k u}  \brk{x}}^{1/k}
 + \frac{\norm{u}_{\lebesgue^\infty\brk{\Omega}}}{\rho} }
 \eqpunct.
\end{equation}
\end{proposition}

Here \(\mathfrak{M} \abs{\Deriv^k u}\colon \Rset^m \to \intvc{0}{+\infty}\) denotes the Hardy--Littlewood maximal function of the function \(\abs{\Deriv^k u} \) defined for each \(x \in \Rset^m\) by
\[
\mathfrak{M} \abs{\Deriv^k u} \brk{x}
= \sup_{r > 0} \frac{1}{\mathcal{L}^m \brk{B_r \brk{x}}} \smashoperator{\int_{\Omega \cap B_r \brk{x}}} \abs{\Deriv^k u}
\]
(note that the set in the measure of the average and the integration domain need not be the same).

\begin{proof}[Proof of \cref{proposition_pointwise_interpolation}]
For almost every \(x \in \Omega\) such that \(B_r \brk{x}\subseteq \Omega\), we have by the Sobolev representation formula \cite{Sobolev_1938}*{\S 7} (see \citelist{\cite{Adams_Fournier_2003}*{Lem. 4.15}\cite{Mazya_2011}*{Thm.\ 1.1.10/1}\cite{Mazya_Poborchi_1997}*{Thm.\ 1.5.1/1}})
\begin{equation}
\label{eq_ung3food9Hah4seixaehaixu}
\begin{split}
 \abs{\Deriv u\brk{x}}
 &\le \frac{\C}{r^{m + 1}} \int_{B_r \brk{x}} \abs{u}
 + \C \int_{B_r \brk{x}} \frac{\abs{\Deriv^k u \brk{y}}}{\abs{x - y}^{m + 1 - k}} \dif y
\end{split}
\end{equation}
and thus
\begin{equation}
\label{eq_bahd3hoh6Tu0chei6phi4EiD}
 \abs{\Deriv u\brk{x}}  \le \C \brk[\Big]{\frac{\norm{u}_{\lebesgue^\infty\brk{\Omega, \Rset^\nu}}}{r}
 + r^{k - 1}\mathfrak{M} \abs{\Deriv^k u} \brk{x}}
 \eqpunct.
\end{equation}
We take \(r \in \intvo{0}{\infty}\) such that
\[
r^k = \min \brk[\Big]{\frac{\norm{u}_{\lebesgue^\infty\brk{\Omega, \Rset^\nu}}}{\mathfrak{M} \abs{\Deriv^k u} \brk{x}}, \rho^k}
\eqpunct,
\]
the conclusion \eqref{eq_thiovohno4xe6lae5Oez4iSh} then follows from \eqref{eq_bahd3hoh6Tu0chei6phi4EiD}.
\resetconstant
\end{proof}

We are now in position to prove \cref{proposition_Wkp_trunc_W1kp_trunc}.

\begin{proof}[Proof of \cref{proposition_Wkp_trunc_W1kp_trunc}]
By \cref{proposition_pointwise_interpolation}, for almost every \(x \in K\), we have
\begin{equation}
\label{eq_Rah3job0yaulah7sieKiu9ei}
\begin{split}
 \abs{\Deriv u \brk{x}}
 &\le \Cl{cst_aech1Aishoh2ukiesol1aqu0} \brk[\Big]{\norm{u}_{\lebesgue^\infty\brk{\Omega, \Rset^\nu}}^{k-1} \mathfrak{M} \brk{\abs{\Deriv^k u}} \brk{x}
 + \frac{\norm{u}_{\lebesgue^\infty\brk{\Omega, \Rset^\nu}}^{k}}{\rho^k}}^{1/k}\\
 &= \Cr{cst_aech1Aishoh2ukiesol1aqu0} \brk[\Big]{\mathfrak{M}  \brk[\Big]{\norm{u}_{\lebesgue^\infty\brk{\Omega, \Rset^\nu}}^{k-1} \abs{\Deriv^k u}
 + \frac{\norm{u}_{\lebesgue^\infty\brk{\Omega, \Rset^\nu}}^{k}}{\rho^k}}\brk{x}}^{1/k}\\
 &\le \Cr{cst_aech1Aishoh2ukiesol1aqu0}  \brk[\Big]{\mathfrak{M}  \brk[\Big]{\norm{u}_{\lebesgue^\infty\brk{\Omega, \Rset^\nu}}^{k-1} \abs{\Deriv^k u}
 + \frac{\norm{u}_{\lebesgue^\infty\brk{\Omega, \Rset^\nu}}^{k}}{\rho^k}  - 2\kappa t^k}_+\brk{x}}^{1/k} + t
 \eqpunct{,}
\end{split}
\end{equation}
where \(\kappa \defeq 1/\brk{2 \Cr{cst_aech1Aishoh2ukiesol1aqu0}^k}\).
If the condition \eqref{eq_IeZohY8amulieB2eemiiB9ph} holds, we deduce from \eqref{eq_Rah3job0yaulah7sieKiu9ei} that
\begin{equation}
\label{eq_CuunuuyiSierebo4choo1uof}
  \brk{\abs{\Deriv u \brk{x}} - t}_+
  \le \C \brk[\big]{\mathfrak{M}  \brk[\big]{\norm{u}_{\lebesgue^\infty\brk{\Omega, \Rset^\nu}}^{k-1} \abs{\Deriv^k u}
   - \kappa t^k}_+\brk{x}}^{1/k}
   \eqpunct{.}
\end{equation}
Integrating \eqref{eq_CuunuuyiSierebo4choo1uof} over \(x \in K\), we have
\begin{equation}
\label{eq_ieSohk8osh8eesixii5fooGh}
\int_{K}  \brk{\abs{\Deriv u} - t}_+^{kp}
 \le \C \int_{K} \brk[\big]{\mathfrak{M}  \brk[\big]{\norm{u}_{\lebesgue^\infty\brk{\Omega, \Rset^\nu}}^{k-1} \abs{\Deriv^k u}
  - \kappa t^k}_+}^{p}
  \eqpunct.
\end{equation}
Since \(p > 1\), thanks to the Hardy--Littlewood maximal function theorem (see for example \citelist{\cite{Stein_1970}*{Thm.\ I.1}\cite{Duoandikoetxea_2001}*{Thm.\ 2.16}}), we have
\begin{equation}
\label{eq_ikohkeimah3quahzee8gooKe}
\begin{split}
 \int_{K} \brk[\big]{\mathfrak{M}  \brk[\big]{\norm{u}_{\lebesgue^\infty\brk{\Omega, \Rset^\nu}}^{k-1} \abs{\Deriv^k u}
 - \kappa t^k}_+}^{p}
 &\le \C \int_{\Omega} \brk[\big]{\norm{u}_{\lebesgue^\infty\brk{\Omega, \Rset^\nu}}^{k-1} \abs{\Deriv^k u}
  - \kappa t^k}_+^{p}\eqpunct,
\end{split}
\end{equation}
and \eqref{eq_aeVoochae8deraegh7iekees} follows from \eqref{eq_ieSohk8osh8eesixii5fooGh} and \eqref{eq_ikohkeimah3quahzee8gooKe}.
\resetconstant
\end{proof}

As an immediate consequence of \cref{proposition_Wkp_trunc_W1kp_trunc}, we get the following relationship between \(\sobolev^{k, p}\)-equi-integrable and \(\sobolev^{1, kp}\)-equi-integrable sequences.

\begin{proposition}
\label{proposition_Wkp_equi_W1kp_equi}
Let \(\manifold{M}\) be a Riemannian manifold, let \(k \in \Nset \setminus \set{0, 1}\) and \(p \in \intvr{1}{\infty}\).
If \(\brk{u_n}_{n \in \Nset}\) is a sequence in \(\sobolev^{k, p} \brk{\manifold{M}, \Rset^{\nu}}\),
if
\[
 \lim_{t \to \infty} \sup_{n \in \Nset} \int_{\manifold{M}} \brk{\abs{\Deriv ^k u_n} - t}_+^p = 0
\]
and if
\[
 \sup_{n \in \Nset} \, \norm{u_n}_{\lebesgue^\infty \brk{\manifold{M}, \Rset^\nu}}< \infty\eqpunct{,}
\]
then
\[
 \lim_{t \to \infty} \sup_{n \in \Nset} \int_{\manifold{M}} \brk{\abs{\Deriv u_n} - t}^{kp}_+ = 0 \eqpunct{.}
\]
\end{proposition}

Thanks to \cref{proposition_Wkp_equi_W1kp_equi}, we are now in position to prove \cref{theorem_Wkp_equiintegrable_strong} when \(p > 1\).

\begin{proof}[Proof of \cref{theorem_Wkp_equiintegrable_strong} for \(p > 1\)]
By \cref{proposition_Wkp_equi_W1kp_equi} and by \cref{theorem_W1p_equiintegrable_strong}, we have
\begin{equation}
\label{eq_heo0Asaev2quuc9eiFoh5Oot}
\soboleh^{k, p}_{\mathrm{Ei}} \brk{\manifold{M}, \manifold{N}}
  \subseteq \soboleh^{1, kp}_{\mathrm{Ei}} \brk{\manifold{M}, \manifold{N}}
  \cap \sobolev^{k, p} \brk{\manifold{M}, \manifold{N}}
  = \soboleh^{1, kp}_{\mathrm{St}} \brk{\manifold{M}, \manifold{N}}
  \cap \sobolev^{k, p} \brk{\manifold{M}, \manifold{N}}
  \eqpunct{.}
\end{equation}
By the characterisation of the strong closure of higher-order Sobolev spaces \cite{Bousquet_Ponce_VanSchaftingen_2501_18149}, we have
\begin{equation}
\label{eq_KieliPoh8Aen2ebeibae3fee}
 \soboleh^{1, kp}_{\mathrm{St}} \brk{\manifold{M}, \manifold{N}}
  \cap \sobolev^{k, p} \brk{\manifold{M}, \manifold{N}}
  =  \soboleh^{k, p}_{\mathrm{St}} \brk{\manifold{M}, \manifold{N}}\eqpunct{.}
\end{equation}
The conclusion \eqref{eq_saah5uCie3xee0fu1eofo1ue} then follows from \eqref{eq_heo0Asaev2quuc9eiFoh5Oot} and \eqref{eq_KieliPoh8Aen2ebeibae3fee} on the one hand and \eqref{eq_yu7Shei1saingoojeic0chie} on the other hand.
\end{proof}

\subsection{The case of \texorpdfstring{\(p=1\)}{p=1}}
When \(p = 1\), the proof above through the Hardy--Littlewood function fails, as there is no reason for \eqref{eq_ikohkeimah3quahzee8gooKe} to hold then.

Because of the embedding \(\sobolev^{k, 1}\brk{\Rset^k, \Rset^\nu} \subseteq \continuous \brk{\Rset^k, \Rset^\nu}\), we get in fact a direct argument that has the advantage of bypassing mappings of vanishing mean oscillation.

The crucial ingredient is the following truncated version of the corresponding Sobolev embedding theorem.

\begin{proposition}
\label{proposition_truncated_Wk1}
For every \(k \in \Nset\setminus \set{0}\) there exists a constant \(C \in \intvo{0}{\infty}\) such that if the set \(\Omega \subseteq \Rset^k\) is open and convex and if \(B_\rho \brk{a}\subseteq \Omega\) with \(a \in \Omega\) and \(\rho \in \intvo{0}{\infty}\),
then for almost every \(x \in \Omega\), we have
\begin{equation}
\label{eq_aph5xuij1Re0Tho3ael2eiw4}
 \abs{u \brk{x}}
 \le C \brk[\bigg]{1 + \frac{\abs{x - a}}{\rho}}^{k - 1}
 \brk[\bigg]{\eta + \int_{\Omega} \brk[\Big]{\abs{\Deriv^k u} - \frac{\eta}{\rho^k}}_+  + \fint_{B_\rho \brk{a}} \brk{\abs{u} - \eta}_+}
 \eqpunct{.}
\end{equation}
\end{proposition}
\begin{proof}
We have, by the Sobolev representation formula (see for example \citelist{\cite{Mazya_2011}*{Thm.\ 1.1.10/1}\cite{Mazya_Poborchi_1997}*{Thm.\ 1.5.1/1}}), for almost every \(x \in \Omega\)
\[
\begin{split}
 \abs{u \brk{x}}
 &\le \C  \brk[\bigg]{1 + \frac{\abs{x - a}}{\rho}}^{k - 1} \brk[\bigg]{\int_{\Omega} \abs{\Deriv^k u} + \fint_{B_\rho \brk{a}}\abs{u}}\\
 &\le \C \brk[\bigg]{1 + \frac{\abs{x - a}}{\rho}}^{k - 1}
 \brk[\bigg]{ \eta + \int_{\Omega} \brk[\Big]{\abs{\Deriv^k u} - \frac{\eta}{\rho^k}}_+  + \fint_{B_\rho \brk{a}} \brk{\abs{u} - \eta}_+}
 \eqpunct{,}
\end{split}
\]
which proves \eqref{eq_aph5xuij1Re0Tho3ael2eiw4}.
\resetconstant
\end{proof}
As a consequence of \cref{proposition_truncated_Wk1}, we obtain the following counterpart of \cref{proposition_W1p_good_grid_measure}.

\begin{proposition}
\label{proposition_Wk1_good_grid_measure}
For every compact Riemannian manifold \(\manifold{N}\) and every \(k \in \Nset \setminus \set{0}\), there exists a constant \(\eta > 0\) such that if \(\Omega \subseteq \Rset^m\) is open and bounded with \(m \ge k\), if \(\varepsilon > \rho > 0\) and if \(u_0, u_1 \in \sobolev^{k, 1}\brk{\Omega, \manifold{N}}\), then
\begin{equation}
\label{eq_phai8iif5aeteekoh6Meix6O}
 \eta \mathcal{L}^m \brk{
 E_{\varepsilon}}
 \le
 \int_{\Omega}
  \brk[\Big]{\abs{\Deriv^k u_0} - \frac{\eta}{\rho}}_+
 + \brk[\Big]{\abs{\Deriv^k u_1} - \frac{\eta}{\rho}}_+
 + \frac{\brk{\abs{u_0 - u_1} - \eta}_+}{\rho^k}\eqpunct{,}
\end{equation}
where
\[
 E_{\varepsilon}
 \defeq
 \set[\big]{\xi \in \Qset^m_\varepsilon \st
 u_0 \restr{K^{\Omega, k}_{\varepsilon, \xi}} \text{ and } u_1 \restr{K^{\Omega, k}_{\varepsilon, \xi}}
 \text{ are not homotopic in \(\continuous \brk{K^{\Omega, k}_{\varepsilon, \xi}, \manifold{N}}\)}}\eqpunct.
\]
\end{proposition}

The set \(K^{\Omega, k}_{\varepsilon, \xi}\) is the \(k\)-dimensional skeleton defined in \eqref{eq_roh8ahga7eishieSh3jie7oo}.
Strictly speaking, in the definition of \(E_\varepsilon\), we
require the maps \(\smash{u_0 \restr{K^{\Omega, k}_{\varepsilon, \xi}}}\) and \(\smash{u_1 \restr{K^{\Omega, k}_{\varepsilon, \xi}}}\) not to be equal almost everywhere to some continuous mappings that are homotopic.

\begin{proof}[Proof of \cref{proposition_Wk1_good_grid_measure}]
We first define the set
\begin{multline*}
 H_\varepsilon \defeq
\bigl\{\xi \in \Qset^m_\varepsilon \st \text{for every \(j \in \set{0, 1}\) and \(\sigma \in  \mathscr{Q}^{m, k}\) such that \(\varepsilon \sigma + \xi \subseteq \Omega\),}\\
 u_j \restr{\varepsilon \sigma + \xi} \in \sobolev^{k, 1}\brk{\varepsilon \sigma + \xi, \manifold{N}} \text{ and }
 \tr_{\partial \brk{\varepsilon \sigma + \xi}} u_j\restr{\varepsilon \sigma + \xi}
 = u_j\restr{\partial \brk{\varepsilon \sigma + \xi}}
 \bigr\}\eqpunct.
\end{multline*}
(Here and in the sequel \(\partial \brk{\varepsilon \sigma + \xi}\) has to be understood as the boundary relative to the \(k\)-dimensional affine plane spanned by \(\varepsilon \sigma + \xi\).)

We have, by Fubini’s theorem,
\begin{equation}
\label{eq_phae5eePh6Fae9leethohw7s}
  \mathcal{L}^m \brk{\Qset^m_\varepsilon \setminus H_{\varepsilon}}= 0\eqpunct.
\end{equation}
Moreover, by the Sobolev embedding theorem, for every \(\xi \in H_\varepsilon\), we have
\(\smash{u_0 \restr{K^{\Omega, k}_{\varepsilon, \xi}}}\) and \(\smash{u_1 \restr{K^{\Omega, k}_{\varepsilon, \xi}}} \in \continuous \brk{K^{\Omega, k}_{\varepsilon, \xi}, \manifold{N}}\) (in the sense they coincide almost everywhere with such a map).

Assuming that \(\manifold{N}\) is isometrically embedded into \(\Rset^\nu\), since \(\manifold{N}\) is compact, there exists some \(\delta \in \intvo{0}{\infty}\) such that if \(X\) is a topological space, if \(v_0\) and \(v_1 \in \continuous \brk{X, \manifold{N}}\) and if \(\norm{v_0 - v_1}_{\lebesgue^\infty \brk{X}}\le \delta\), then the maps \(v_0\) and \(v_1\) are homotopic.
If we define the set
\[
 F_{\varepsilon}
 \defeq
 \set[\big]{\xi \in H_\varepsilon \st
 \norm{u_0 \restr{K^{\Omega, k}_{\varepsilon, \xi}} -u_1 \restr{K^{\Omega, k}_{\varepsilon, \xi}}}_{\lebesgue^\infty \brk{K^{\Omega, k}_{\varepsilon, \xi},\Rset^\nu}}\ge \delta}\eqpunct,
\]
we have then
\[
 E_{\varepsilon} \cap H_\varepsilon \subseteq F_{\varepsilon}
\]
and thus
\begin{equation}
\label{eq_Gahgich5keekuu6ooquohche}
 E_{\varepsilon} \subseteq F_\varepsilon \cup \brk{\Qset^m_\varepsilon \setminus H_\varepsilon} \eqpunct.
\end{equation}

On the other hand, if \(\sigma \in  \mathscr{Q}^{m, k}\) and \(\varepsilon \sigma + \xi \subseteq \Omega\), if \(u_0\restr{\varepsilon \sigma + \xi} \) and \(u_1 \restr{\varepsilon \sigma + \xi} \in \sobolev^{k, 1}\brk{\varepsilon \sigma + \xi, \manifold{N}}\), then by \cref{proposition_truncated_Wk1}, we have
\begin{equation}
\label{eq_ieCh8la0aeHohsooco9sa5ka}
\begin{split}
  &\norm{u_0\restr{\varepsilon \sigma + \xi} - u_1 \restr{\varepsilon \sigma + \xi}}_{\lebesgue^\infty \brk{\varepsilon \sigma + \xi, \Rset^\nu}}\\
  &\qquad \le \Cl{cst_Eim9aingahhathokai7eith9} \brk[\bigg]{\eta + \int_{\varepsilon \sigma + \xi}  \brk[\Big]{\abs{\Deriv^k u_0} - \frac{\eta}{\rho}}_+
 + \brk[\Big]{\abs{\Deriv^k u_1} - \frac{\eta}{\rho}}_+
 + \frac{\brk{\abs{u_0 - u_1} - \eta}_+}{\rho^k}}\eqpunct.
\end{split}
\end{equation}
By the definition of \(K^{\Omega, k}_{\varepsilon, \xi}\) in \eqref{eq_roh8ahga7eishieSh3jie7oo}, it follows from \eqref{eq_ieCh8la0aeHohsooco9sa5ka} that if \(\Cr{cst_Eim9aingahhathokai7eith9}\eta \le \delta/2\),
\begin{equation}
\label{eq_IequaengoeKaing8At7eesho}
\begin{split}
\mathcal{L}^m \brk{F_\varepsilon}
  &\le \Cl{cst_she6vi9oNgae7giez2Coidee}
 \int_{\Qset^m_\varepsilon} \int_{K^{\Omega, k}_{\varepsilon, \xi}}
  \brk[\Big]{\abs{\Deriv^k u_0} - \frac{\eta}{\rho}}_+
 + \brk[\Big]{\abs{\Deriv^k u_1} - \frac{\eta}{\rho}}_+
 + \frac{\brk{\abs{u_0 - u_1} - \eta}_+}{\rho^k}\\
 &\le \Cr{cst_she6vi9oNgae7giez2Coidee}\tbinom{m}{k}
 \int_{\Omega}
  \brk[\Big]{\abs{\Deriv^k u_0} - \frac{\eta}{\rho}}_+
 + \brk[\Big]{\abs{\Deriv^k u_1} - \frac{\eta}{\rho}}_+
 + \frac{\brk{\abs{u_0 - u_1} - \eta}_+}{\rho^k}\eqpunct.
 \end{split}
\end{equation}
We deduce \eqref{eq_phai8iif5aeteekoh6Meix6O} from \eqref{eq_phae5eePh6Fae9leethohw7s} and \eqref{eq_IequaengoeKaing8At7eesho} in view of \eqref{eq_Gahgich5keekuu6ooquohche}.
\end{proof}

As a consequence of \cref{proposition_Wk1_good_grid_measure}, we get the counterpart of \cref{proposition_Wk1_equiintegrable_strong_local} for \(\soboleh^{k, 1}_{\mathrm{Ei}} \brk{\Omega, \manifold{N}}\).

\begin{proposition}
\label{proposition_Wk1_equiintegrable_strong_local}
If the set \(\Omega \subset \Rset^m\) is open and bounded with a smooth boundary, if \(\manifold{N}\) is a compact Riemannian manifold and if \(k \in \set{1, \dotsc, m - 1}\), then
\begin{equation}
\label{eq_deiyoh1wo1gahZaich1FuH3a}
  \soboleh^{k, 1}_{\mathrm{St}} \brk{\Omega, \manifold{N}}
 =
 \soboleh^{k, 1}_{\mathrm{Ei}} \brk{\Omega, \manifold{N}}
 \eqpunct{.}
\end{equation}
\end{proposition}

\begin{proof}
This is a consequence of \cref{proposition_Wk1_good_grid_measure} and of the strong approximability criterion for Sobolev mappings \cite{Bousquet_Ponce_VanSchaftingen_2501_18149}.
\end{proof}

We are now in position to prove \cref{theorem_Wkp_equiintegrable_strong} for \(p=1\).

\begin{proof}[Proof of \cref{theorem_Wkp_equiintegrable_strong} for \(p = 1\)]
Let \(u \in \soboleh^{k, 1}_{\mathrm{Ei}} \brk{\manifold{M}, \manifold{N}}\).
From \cref{proposition_Wk1_equiintegrable_strong_local} and a local chart argument, we have bounded open sets \(\Omega_1, \dotsc, \Omega_r \subseteq \manifold{M}\) with a smooth boundary such that \(\manifold{M} \subseteq \bigcup_{i = 1}^r \Omega_i\) and \(\smash{u\restr{\Omega_i} \in \soboleh^{k, 1}_{\mathrm{Ei}} \brk{\Omega_i, \manifold{N}}} = \smash{\soboleh^{k, 1}_{\mathrm{St}} \brk{\Omega_i, \manifold{N}}}\).

By the Gagliardo-Nirenberg inequality \citelist{\cite{Nirenberg_1959}\cite{Gagliardo_1958}\cite{Gagliardo_1959}} (see also \citelist{\cite{Leoni_2017}*{Thm.\ 12.85}\cite{Brezis_Mironescu_2018}}), we have
\(u \in \smash{\soboleh^{1, k}_{\mathrm{Bd}} \brk{\manifold{M}, \manifold{N}}}\).
Assuming that \(\manifold{M}\) is isometrically embedded in \(\Rset^\mu\) and that \(\smash{\Pi_{\manifold{M}}\colon \manifold{M} + B_\delta\to \manifold{M}}\) is a smooth retraction,
for almost every \(\xi \in B_\delta\), we have \(u \compose \brk{\Pi_\manifold{M} + \xi}\restr{\manifold{M}^{k - 1}} = V_\xi \restr{\manifold{M}^{k-1}}\) almost everywhere in \(\smash{\manifold{M}^{k-1}}\), with \(V_\xi \in \continuous \brk{\manifold{M}, \manifold{N}}\)  \cite{Hang_Lin_2003_II}*{Thm.\ 7.1}. Isobe’s criterion of global strong approximability from local strong approximability \cite{Isobe_2005}*{Thm.\ 1.6} implies that \(u \in \smash{\soboleh^{1, k}_{\mathrm{St}}} \brk{\manifold{M}, \manifold{N}}\).
By the global strong approximability criterion for Sobolev mappings \cite{Bousquet_Ponce_VanSchaftingen_2501_18149} we conclude that \(u \in \smash{\soboleh^{k, 1}_{\mathrm{St}} \brk{\manifold{M}, \manifold{N}}}\).
\end{proof}

\section{Fractional Sobolev spaces}

In fractional spaces, equi-integrable and strong convergence coincide so that the generalization of \cref{theorem_W1p_equiintegrable_strong} is trivial, and the question of equi-integrable approximation reduces to the strong approximation problem which has been studied for \(0 < s < 1\) by Brezis and Mironescu \cite{Brezis_Mironescu_2015} and for \(s > 1\) by Detaille \cite{Detaille_2305_12589}.

\begin{proposition}
\label{proposition_equivalence_fractional}
Let \(\manifold{M}\) be an \(m\)-dimensional compact Riemannian manifold,
let \( k \in \Nset\), \(s \in \intvo{0}{1}\) and \(\brk{u_n}_{n \in \Nset}\) be a sequence in \(\sobolev^{k + s, p}\brk{\manifold{M}, \Rset^\nu}\).
The following are equivalent:
\begin{enumerate}[label=(\roman*)]
 \item
 \label{it_opeihi2eicas1ooy4Eduthau}
 \(u_n \to u\) in \(\sobolev^{k + s, p}\brk{\manifold{M}, \Rset^\nu}\),
 \item
 \label{it_Queovohxei5Ohxoowothahr3}
 \(u_n \to u\) in \(\lebesgue^p \brk{\manifold{M}, \Rset^\nu}\) and
 \[
  \lim_{\rho \to 0} \sup_{n \in \Nset}
  \smashoperator{\iint_{\substack{\brk{x, y}\in \manifold{M}\times \manifold{M}\\ d \brk{x, y} <\rho}}}
  \frac{\abs{\Deriv^k u_n \brk{x} - \Deriv^k u_n \brk{y}}^p}{\abs{x - y}^{m + s p}} \dif x \dif y = 0 \eqpunct{,}
 \]
  \item
  \label{it_nieVaN5iep0Jah7gahnaiTei}
 \(u_n \to u\) in \(\lebesgue^p \brk{\manifold{M}, \Rset^\nu}\) and
 \[
  \lim_{\rho \to 0} \limsup_{n \to \infty}
  \smashoperator{\iint_{\substack{\brk{x, y}\in \manifold{M}\times \manifold{M}\\ d \brk{x, y} <\rho}}}
  \frac{\abs{\Deriv^k u_n \brk{x} - \Deriv^k u_n \brk{y}}^p}{\abs{x - y}^{m + s p}} \dif x \dif y = 0 \eqpunct{.}
 \]
\end{enumerate}
\end{proposition}

The first ingredient of the proof of \cref{proposition_equivalence_fractional} is the following localization estimate of the difference in norm.

\begin{proposition}
\label{proposition_Wsp_strong_equi_integrable}
For every \(m \in \Nset \setminus \set{0}\), \(s \in \intvo{0}{1}\) and \(p \in \intvr{1}{\infty}\), there exists a constant \(C \in \intvo{0}{\infty}\) such that if the functions \(u_0\) and \(u_1 \colon \Omega \to \Rset^\nu\),  are measurable, one has
\begin{multline}
\label{eq_uojeiJees7aiju5ahng2gaoh}
 \smashoperator[r]{\iint_{\Omega \times \Omega}}
\frac{\abs{\brk{u_0 - u_1} \brk{x} - \brk{u_0 - u_1}\brk{y}}^p}{\abs{x - y}^{m + sp}} \dif x \dif y\\
\le C \brk[\bigg]{
\smashoperator[r]{\iint_{\substack{\brk{x, y}\in \Omega \times \Omega\\ 
\abs{x - y}\le \rho}}}
\frac{\abs{u_0 \brk{x} - u_0 \brk{y}}^p + \abs{u_1 \brk{x} - u_1 \brk{y}}^p}{\abs{x - y}^{m + sp}} \dif x \dif y
+ 
\int_{\Omega} \frac{\abs{u_0 - u_1}^p}{\rho^{sp}}
}\eqpunct.
\end{multline}
\end{proposition}

Although not stated in this form, the proof of \cref{proposition_Wsp_strong_equi_integrable} is a fundamental ingredient of the continuity on fractional Sobolev spaces of the composition with Lipschitz-continuous mappings in \cite{Bourgain_Brezis_Mironescu_2004}*{claim (5.43)}
(see also \cite{Bousquet_Ponce_VanSchaftingen_2014}*{Lem.\ 2.5}).

\begin{proof}[Proof of \cref{proposition_Wsp_strong_equi_integrable}]
We have on the one hand 
\begin{multline}
\label{eq_aighu9lohZoxais0oej5Yi0r}
 \smashoperator[r]{\iint_{\substack{\brk{x, y}\in \Omega \times \Omega\\ 
\abs{x - y}\le \rho}}}
\frac{\abs{\brk{u_0 - u_1} \brk{x} - \brk{u_0 - u_1}\brk{y}}^p}{\abs{x - y}^{m + sp}} \dif x \dif y \\
\le 2^{p - 1}
\smashoperator{
\iint_{\substack{\brk{x, y}\in \Omega \times \Omega\\ 
\abs{x - y}\le \rho}}}
\frac{\abs{u_0 \brk{x} - u_0 \brk{y}}^p + \abs{u_1 \brk{x} - u_1 \brk{y}}^p}{\abs{x - y}^{m + sp}} \dif x \dif y
\eqpunct,
\end{multline}
and on the other hand 
\begin{equation}
\label{eq_eo7ooc5ahngoodie0aijieFi}
\begin{split}
 &\smashoperator[r]{
  \iint_{\substack{\brk{x, y}\in \Omega \times \Omega\\ 
\abs{x - y}> \rho}}}
\frac{\abs{\brk{u_0 - u_1} \brk{x} - \brk{u_0 - u_1}\brk{y}}^p}{\abs{x - y}^{m + sp}} \dif x \dif y \\
&\qquad 
\le 2^{p - 1}
\smashoperator{\iint_{\substack{\brk{x, y}\in \Omega \times \Omega\\ 
\abs{x - y}> \rho}}}
\frac{\abs{u_0 \brk{x} - u_1 \brk{x}}^p + \abs{u_0 \brk{y} - u_1 \brk{y}}^p}{\abs{x - y}^{m + sp}} \dif x \dif y\\
&\qquad 
= 2^{p}
\smashoperator{\iint_{\substack{\brk{x, y}\in \Omega \times \Omega\\ 
\abs{x - y}> \rho}}}
\frac{\abs{u_0 \brk{x} - u_1 \brk{x}}^p}{\abs{x - y}^{m + sp}} \dif x \dif y
\eqpunct.
\end{split}
\end{equation}
We then compute
\begin{equation}
\label{eq_gaqu8Quafaih1ub7iekaphoh}
\smashoperator{\iint_{\substack{\brk{x, y}\in \Omega \times \Omega\\ 
\abs{x - y}> \rho}}}
\frac{\abs{u_0 \brk{x} - u_1 \brk{x}}^p}{\abs{x - y}^{m + sp}} \dif x \dif y
\le \frac{\C}{\rho^{sp}} \int_{\Omega}\abs{u_0 - u_1}^p\eqpunct.
\end{equation}
The conclusion \eqref{eq_uojeiJees7aiju5ahng2gaoh} then follows from \eqref{eq_aighu9lohZoxais0oej5Yi0r}, \eqref{eq_eo7ooc5ahngoodie0aijieFi} and \eqref{eq_gaqu8Quafaih1ub7iekaphoh}.
\resetconstant
\end{proof}

The other ingredient is the following global control of fractional energy by its localisation at small scales.

\begin{proposition}
\label{proposition_fractional_rescaling}
If \(\Omega \subseteq \Rset^m\) is convex, \(s \in \intvo{0}{1}\), \(p \in \intvr{1}{\infty}\), \(u \colon \Omega \to \Rset^\nu\) is Borel-measurable, \(\rho \in \intvo{0}{\infty}\) and \(\ell \in \Nset \setminus \set{0}\), then
\begin{equation}
\label{eq_Veequ5ra2jeomeehiezoe1ie}
 \smashoperator[r]{\iint_{\substack{\brk{x, y} \in \Omega\times \Omega \\ \abs{x - y} \le \ell \rho}}} \frac{\abs{u \brk{x} - u \brk{y}}^p}{\abs{x - y}^{m + sp}} \dif x \dif y
 \le \ell^{\brk{1 - s}p}
 \smashoperator{\iint_{\substack{\brk{x, y} \in \Omega \times \Omega \\ \abs{x - y} \le \rho}}} \frac{\abs{u \brk{x} - u \brk{y}}^p}{\abs{x - y}^{m + sp}} \dif x \dif y\eqpunct.
\end{equation}
\end{proposition}
\begin{proof}
We have, by the triangle inequality for every \(x, y \in \Omega\),
\[
 \abs{u \brk{x} - u \brk{y}}
 \le
 \sum_{j = 1}^{\ell} \abs{u \brk{\tfrac{\ell - j + 1}{\ell} x + \tfrac{j - 1}{\ell} y} - u \brk{\tfrac{\ell-j}{\ell} x + \tfrac{j}{\ell} y}}
 \eqpunct{,}
\]
and thus by convexity and a change of variable
\[
\begin{split}
  &\smashoperator[r]{\iint_{\substack{\brk{x, y} \in \Omega \times \Omega\\ \abs{x - y} \le \ell \rho}}} \frac{\abs{u \brk{x} - u \brk{y}}^p}{\abs{x - y}^{m + sp}} \dif x \dif y\\
  &\qquad \le \ell^{p - 1} \sum_{j = 1}^\ell
  \smashoperator[r]{\iint_{\substack{\brk{x, y} \in \Omega \times \Omega\\ \abs{x - y} \le \ell \rho}}} \frac{ \abs{u \brk{\tfrac{\ell - j + 1}{\ell} x + \tfrac{j - 1}{\ell} y} - u \brk{\tfrac{\ell-j}{\ell} x + \tfrac{j}{\ell} y}}^p}{\abs{x - y}^{m + sp}} \dif x \dif y\\
  &\qquad \le \ell^{p - 1}\sum_{j = 1}^{\ell} \ell^{- sp}
 \smashoperator{\iint_{\substack{\brk{x, y} \in \Omega \times \Omega\\ \abs{x - y} \le \rho}}} \frac{\abs{u \brk{x} - u \brk{y}}^p}{\abs{x - y}^{m + sp}} \dif x \dif y\eqpunct,
\end{split}
\]
and \eqref{eq_Veequ5ra2jeomeehiezoe1ie} follows.
\end{proof}

Thanks to \cref{proposition_fractional_rescaling} and \cref{proposition_Wsp_strong_equi_integrable}, we can now prove \cref{proposition_equivalence_fractional}.

\begin{proof}[Proof of \cref{proposition_equivalence_fractional}]
If \ref{it_opeihi2eicas1ooy4Eduthau}, then \ref{it_Queovohxei5Ohxoowothahr3} follows from a classical equi-integrability argument. 
The condition \ref{it_nieVaN5iep0Jah7gahnaiTei} is clearly weaker than \ref{it_Queovohxei5Ohxoowothahr3}.

Conversely, if \ref{it_nieVaN5iep0Jah7gahnaiTei} holds, then \cref{proposition_fractional_rescaling} implies that the sequence \(\brk{u_n}_{n \in \Nset}\) is bounded in \(\sobolev^{k + s, p}\brk{\manifold{M}, \Rset^\nu}\) and the classical fractional Rellich-Kondrashov compactness theorems show that \(u \in \sobolev^{k, p}\brk{\manifold{M}, \Rset^\nu}\) and that \(\brk{u_n}_{n \in \Nset}\) converges to \(u\) strongly in \(\sobolev^{k, p}\brk{\manifold{M}, \Rset^\nu}\).
It then follows from \cref{proposition_Wsp_strong_equi_integrable} and Fatou’s lemma that
\begin{equation}
\label{eq_uquaimiej6Eitie2wae8Xuis}
\begin{split} 
&
\limsup_{n \to \infty} \smashoperator{\iint_{\manifold{M} \times \manifold{M}}}
\frac{\abs{\brk{\Deriv^k u_n - \Deriv^k u} \brk{x} - \brk{\Deriv^k u_n - \Deriv^k u}\brk{y}}^p}{d \brk{x , y}^{m + sp}} \dif x \dif y\\
&\qquad \le \Cl{cst_eiPhaaCohth8gi5vaiC1ni6O} \bigg(\limsup_{n \to \infty} 
\smashoperator{\iint_{\substack{\brk{x, y}\in \manifold{M} \times \manifold{M}\\ 
d\brk{x, y}\le \rho}}}
\frac{\abs{\Deriv^k u_n \brk{x} - \Deriv^k u_n \brk{y}}^p+\abs{\Deriv^k u \brk{x} - \Deriv^k u \brk{y}}^p}{d\brk{x, y}^{m + sp}} \dif x \dif y\\[-1em]
&\hspace{20em}
+ \limsup_{n \to \infty}
\int_{\manifold{M}} \frac{\abs{\Deriv^k u_n - \Deriv^k u}^p}{\rho^{sp}}\biggr)
\\ 
&\qquad \le 2\Cr{cst_eiPhaaCohth8gi5vaiC1ni6O}  \limsup_{n \to \infty} 
\smashoperator{\iint_{\substack{\brk{x, y}\in \manifold{M} \times \manifold{M}\\ 
d\brk{x, y}\le \rho}}}
\frac{\abs{\Deriv^k u_n \brk{x} - \Deriv^k u_n \brk{y}}^p}{d\brk{x, y}^{m + sp}} \dif x \dif y
\eqpunct{,}
\end{split}
\end{equation}
Letting \(\rho \to 0\) in \eqref{eq_uquaimiej6Eitie2wae8Xuis}, we get the conclusion \ref{it_opeihi2eicas1ooy4Eduthau}.
\end{proof}

\section{Equi-integrability and cohomology}
\label{section_jacobian}

Given differential form \(\omega \in \smooth^\infty \brk{\manifold{N}, \bigwedge^\ell \tangent^* \manifold{N}}\) and a smooth mapping \(u \in \smooth^\infty \brk{\manifold{M}, \manifold{N}}\), one has
\begin{equation}
\label{eq_ishohPae5Oso6caewohxaeji}
 \extdiff \brk{u^* \omega} = u^* \brk{\extdiff \omega}\eqpunct{,}
\end{equation}
where \(d\) is the exterior differential and \(u^*\omega\) is the pull-back of \(\omega\) by \(u\).
If \(\omega\) is closed, that is, if \(\extdiff \omega = 0\), one has from \eqref{eq_ishohPae5Oso6caewohxaeji}
\begin{equation}
\label{eq_aumai3ahfei2ooPaeweiPoh1}
 \extdiff \brk{u^* \omega} = 0
 \eqpunct{.}
\end{equation}
If \(p \ge \ell\) and if \(u \in \sobolev^{1, p}\brk{\manifold{M}, \manifold{N}}\), then the pull-back \(u^* \omega\) is still defined in \(\lebesgue^{p/\ell} \brk{\manifold{M}, \bigwedge^\ell \tangent^* \manifold{M}}\) and satisfies
\[
\abs{u^* \omega} \le C \abs{\Deriv u}^\ell
\eqpunct{,}
\]
but, if \(p < \ell + 1\) or, equivalently, \(\ell = \floor{p}\), the pull-back \(u^* \brk{\extdiff \omega}\) is not well-defined as a function and so \eqref{eq_ishohPae5Oso6caewohxaeji} does not make sense anymore.
One can however still interpret the condition \eqref{eq_aumai3ahfei2ooPaeweiPoh1} in \emph{the sense of distributions,} by requiring that for every compactly supported test form \(\varphi \in \smooth^\infty_c \brk{\manifold{M}, \bigwedge^{m - \ell - 1} \tangent^* \manifold{M}}\),
\begin{equation}
\label{eq_eineixae2phi9ea8ie4Uthie}
  \int_{\manifold{M}} u^* \omega \wedge \extdiff \varphi = 0
  \eqpunct.
\end{equation}
(A partition of unity argument shows that the support of \(\varphi\) can be restricted to small balls so that the domain manifold \(\manifold{M}\) need not be orientable.)
On the other hand, if \(u \in \sobolev^{1, \ell}\brk{\manifold{M}, \manifold{N}}\), if \(\theta \in \smooth^\infty \brk{\manifold{N}, \bigwedge^{\ell - 1} \tangent^* \manifold{N}}\) and if \(\varphi \in \smooth^\infty_c \brk{\manifold{M}, \bigwedge^{m - \ell - 1} \tangent^* \manifold{M}}\),
\begin{equation}
\label{eq_pheunei9AiG2xiereemahqu9}
  \int_{\manifold{M}} u^* \brk{\extdiff \theta} \wedge \extdiff \varphi = 0
  \eqpunct.
\end{equation}
(Extending \(\theta\) to \(\Rset^\nu\), \eqref{eq_pheunei9AiG2xiereemahqu9} can be proved for \(u \in \sobolev^{1, p}\brk{\manifold{M}, \Rset^\nu}\) by an approximation argument.)
In view of \eqref{eq_pheunei9AiG2xiereemahqu9}, the condition \eqref{eq_eineixae2phi9ea8ie4Uthie} has to be checked for every \(\omega \in \smooth^\infty \brk{\manifold{N}, \bigwedge^{\ell} \tangent^* \manifold{N}}\) satisfying \(\extdiff \omega = 0\) modulo \(\extdiff \theta\) for \(\theta \in \smooth^\infty \brk{\manifold{N}, \bigwedge^{\ell - 1} \tangent^* \manifold{N}}\); in other words it only depends on the \emph{de Rham cohomology} \(\soboleh^\ell \brk{\manifold{N}}\) defined as the quotient of the space of closed forms in \(\smooth^\infty \brk{\manifold{N}, \bigwedge^{\ell} \tangent^* \manifold{N}}\) by the subspace of exact forms.
In the particular case where \(\manifold{N} = \Sset^\ell\), the condition \eqref{eq_eineixae2phi9ea8ie4Uthie} holds for every closed form \(\omega \in \smooth^\infty \brk{\manifold{N}, \bigwedge^{\ell} \tangent^* \manifold{N}}\) if and only if it holds for the canonical volume form.

The condition \eqref{eq_pheunei9AiG2xiereemahqu9} can be used to describe the strong closure of smooth maps \(\soboleh^{1, p}_{\mathrm{St}} \brk{\manifold{M}, \manifold{N}}\) for some target manifolds.

\begin{theorem}
\label{theorem_cohomology}
Let \(\manifold{M}\) and \(\manifold{N}\) be compact Riemannian manifolds and let \(p \in \intvr{1}{\infty}\).
Assume that every \(f \in \smooth^\infty \brk{\Sset^{\floor{p}}, \manifold{N}}\) satisfying for every \(\omega \in \smooth^\infty \brk{\manifold{N}, \bigwedge^{\floor{p}} \tangent^* \manifold{N}}\) such that \(\extdiff \omega = 0\) the condition
\[
 \int_{\Sset^{\floor{p}}} f^* \omega = 0
 \eqpunct,
\]
is homotopic to a constant and that for every \(g \in \continuous \brk{\manifold{M}^{\floor{p}}, \manifold{N}}\) one has \(G\restr{\manifold{M}^{\floor{p - 1}}} = g \restr{\manifold{M}^{\floor{p - 1}}}\) for some \(G \in \continuous \brk{\manifold{M}, \manifold{N}}\).
Then, one has
\[
 u\in
 \soboleh^{1, p}_{\mathrm{St}} \brk{\manifold{M}, \manifold{N}}
\]
if and only if \(u \in \sobolev^{1, p}\brk{\manifold{M}, \manifold{N}}\) and for every \(\omega \in \smooth^\infty \brk{\manifold{N}, \bigwedge^{\floor{p}} \tangent^* \manifold{N}}\) such that
\( \extdiff \omega = 0\) on \(\manifold{N}\) and every \(\varphi \in  \smooth^\infty \brk{\manifold{M}, \bigwedge^{m - \floor{p} - 1} \tangent^* \manifold{M}}\) with orientable support,
one has
\[
 \int_{\manifold{M}} u^* \omega \wedge \extdiff \varphi = 0
 \eqpunct{.}
\]
\end{theorem}

The condition on \(f \in \smooth^\infty \brk{\Sset^{\floor{p}}, \manifold{N}}\) in \cref{theorem_cohomology} requires the Hurewicz homomorphism \(\smash{\pi_{\floor{p}} \brk{\manifold{N}}} \to \smash{H_{\floor{p}} \brk{\manifold{N}, \Qset}}\) to be a monomorphism; it is equivalent to having \(\smash{\pi_{\floor{p}} \brk{\manifold{N}} }\to \smash{H_{\floor{p}} \brk{\manifold{N}, \Zset}}\) injective and \(\smash{\pi_{\floor{p}} \brk{\manifold{N}}}\) torsion-free.
\Cref{theorem_cohomology} is due to Bethuel when \(p = 2\) and \(\manifold{N} = \Sset^2\) \cite{Bethuel_1990}, to Demengel when \(1\le p < 2\) and \(\manifold{N} = \Sset^1\) \cite{Demengel_1990} (see also \cite{Brezis_Mironescu_2021}*{Theorem 10.4}), and to Bethuel, Coron, Demengel and Hélein in general \cite{Bethuel_Coron_Demengel_Helein_1991};
the global condition in the statement of \cref{theorem_cohomology} comes from Hang and Lin's work \cite{Hang_Lin_2003_II}.

\begin{theorem}
\label{theorem_equi_integrable_cohomotology}
Let \(\manifold{M}\) and \(\manifold{N}\) be compact Riemannian manifolds and let \(p \in \intvr{1}{\infty}\).
If
\[
 u \in \soboleh^{1, p}_{\mathrm{Ei}} \brk{\manifold{M}, \manifold{N}}
 \eqpunct,
\]
then for every \(\omega \in \smooth^\infty \brk{\manifold{N}, \bigwedge^{\floor{p}} \tangent^* \manifold{N}}\) such that \(\extdiff \omega = 0\) on \(\manifold{N}\) and for every \(\varphi \in  \smooth^\infty \brk{\manifold{M}, \bigwedge^{m - \floor{p} - 1} \tangent^* \manifold{M}}\) with orientable support, one has
\begin{equation}
\label{eq_zohzuPaqu6lae1sahquo1Loo}
 \int_{\manifold{M}} u^* \omega \wedge \extdiff \varphi = 0
 \eqpunct{.}
\end{equation}
\end{theorem}

When \(p = 1\) and when the target is the unit circle \(\manifold{N} = \Sset^1\),  \cref{theorem_equi_integrable_cohomotology} is due to Hang \cite{Hang_2002}*{First proof of Example 2.1}.

Under the assumptions of \cref{theorem_cohomology}, any of \cref{theorem_W1p_equiintegrable_strong,theorem_cohomology,theorem_equi_integrable_cohomotology} is a direct consequence of the two other ones.

\Cref{theorem_equi_integrable_cohomotology} will be a consequence of the following equi-integrable continuity result for pullbacks.

\begin{proposition}
\label{proposition_equiintegrable_limit_pullback}
Let \(\Omega \subset \Rset^m\) be an open set, let \(\nu \in \Nset\), let \(p \in \intvr{1}{\infty}\) and let \(\brk{u_n}_{n \in \Nset}\) be a sequence in \(\sobolev^{1, p}\brk{\Omega, \Rset^\nu}\) and let \(u \colon \Omega \to \Rset^\nu\) be measurable.
If \(\ell \le p\), if
\[
 \lim_{n \to \infty} \int_{\Omega}\abs{u_n - u} = 0
\]
and if
\begin{equation}
\label{eq_EiKoevaiPazoo5sit7eiquaz}
 \lim_{t \to \infty}
 \sup_{n \in \Nset}
 \int_{\Omega} \brk{\abs{\Deriv u_n} - t}_+^p = 0
 \eqpunct,
\end{equation}
then for every \(\varphi \in \continuous^\infty_c \brk{\Omega, \bigwedge^{m - \ell} \Rset^m }\) and every \(\omega \in \continuous \brk{\Rset^\nu, \bigwedge^\ell \Rset^\nu} \cap \lebesgue^\infty \brk{\Rset^\nu, \bigwedge^\ell \Rset^\nu}\),
one has
\begin{equation}
\label{eq_ief1Aef1umai2ahG7zo4ohh5}
\lim_{n \to \infty}
 \int_{\Omega} u_n^* \omega \wedge \varphi
 =
 \int_{\Omega}
 u^* \omega \wedge \varphi\eqpunct{.}
\end{equation}
\end{proposition}

The first tool for the proof of \cref{proposition_equiintegrable_limit_pullback}
 is the following classical continuity of exterior products for bounded convergence in \(\sobolev^{1, p}\) \cite{Reshetnyak_1967}*{Thm.\ 2} (see also \citelist{\cite{Reshetnyak_1968}*{Thm.\ 4} \cite{Wente_1969}*{Thm.\ 3.5}\cite{Giaquinta_Modica_Soucek_1998}*{\S 3.3.1}}).

\begin{proposition}
\label{proposition_jacobian_measure}
Let \(\Omega \subset \Rset^m\) be an open set, let \(\nu \in \Nset\), let \(p \in \intvr{1}{\infty}\) and let \(\brk{u_n}_{n \in \Nset}\) be a sequence in \(\sobolev^{1, p}\brk{\Omega, \Rset^\nu}\) and let \(u \colon \Omega \to \Rset^\nu\) be measurable.
If \(\ell \le p\), if
\[
 \lim_{n \to \infty} \int_{\Omega}\abs{u_n - u} = 0
\]
and if
\[
  \sup_{n \in \Nset} \int_{\Omega}\abs{\Deriv u_n}^p < \infty \eqpunct,
\]
then for every \(\varphi \in \continuous_c \brk{\Omega, \bigwedge^{m - \ell} \Rset^m}\) and for every \(h_1, \dotsc, h_\ell \in \Rset^m\)
\[
\lim_{n \to \infty}
  \int_{\Omega} \partial_{h_1} u_n \wedge \dotsb \wedge \partial_{h_\ell} u_n \wedge \varphi
=
\int_{\Omega} \partial_{h_1} u \wedge \dotsb \wedge \partial_{h_\ell} u \wedge \varphi
\eqpunct{.}
\]
\end{proposition}

The other tool to prove \cref{proposition_equiintegrable_limit_pullback} is the following stability of the weak convergence for equi-integrable sequences.

\begin{lemma}
\label{lemma_distrib_to_L1_weak}
Let \(\Omega \subseteq \Rset^m\) be open, let \(f \in \lebesgue^1 \brk{\Omega, \Rset}\) and
let \(\brk{f_n}_{n \in \Nset}\) be a sequence in \(\lebesgue^1 \brk{\Omega, \Rset}\) such that for every \(\varphi \in \continuous_c \brk{\Omega, \Rset}\),
\[
 \lim_{n \to \infty} \int_{\Omega} f_n \varphi  = \int_{\Omega} f \varphi
\]
and
\begin{equation}
\label{eq_MaeGhoojaige4eipeesh9Ag0}
 \lim_{t \to \infty} \sup_{n \in \Nset} \int_{\Omega} \brk{\abs{f_n} - t}_+ = 0
 \eqpunct{,}
\end{equation}
then for every sequence \(\brk{g_n}_{n \in \Nset}\) in \(\lebesgue^1 \brk{\Omega, \Rset} \cap \lebesgue^\infty \brk{\Omega, \Rset}\) satisfying
\begin{align}
\label{eq_OoroovieJah8eshuj0ohf1oa}
\lim_{n \to \infty} \int_{\Omega} \abs{g_n - g} &= 0&
&\text{and}&
\sup_{n \in \Nset}  \, \norm{g_n}_{\lebesgue^\infty \brk{\Omega}} <\infty
\eqpunct,
\end{align}
one has
\[
\lim_{n \to \infty} \int_{\Omega} f_n \,g_n = \int_{\Omega} f g
\eqpunct.
\]
\end{lemma}
\begin{proof}
By \eqref{eq_OoroovieJah8eshuj0ohf1oa}, we have \(g \in  \lebesgue^1 \brk{\Omega, \Rset} \cap \lebesgue^\infty \brk{\Omega, \Rset}\) and there exists a sequence \(\brk{\varphi_\ell}_{\ell \in \Nset}\) in \(C^\infty_c \brk{\Omega, \Rset}\) such that
\begin{align*}
 \lim_{\ell \to \infty} \int_{\Omega} \abs{\varphi_\ell - g} &= 0&
 &\text{and}&
 \sup_{n \in \Nset}\, \norm{\varphi_\ell - g}_{\lebesgue^\infty \brk{\Omega}} &< \infty
 \eqpunct.
\end{align*}
Since \(f \in \lebesgue^1 \brk{\Omega, \Rset}\) and since \eqref{eq_MaeGhoojaige4eipeesh9Ag0} holds,
given \(\varepsilon > 0\), there exists
\(t \in \intvo{0}{\infty}\) such that if \(n \in \Nset\) is large enough,
\begin{align}
\label{eq_aiv4viejiaJ2sae1quohNg1u}
 \int_{\Omega} \brk{\abs{f_n}- t}_+ &\le \varepsilon&
&\text{ and }&
\int_{\Omega} \brk{\abs{f} - t}_+ &\le \varepsilon
\eqpunct.
\end{align}
We have then, for \(n \in \Nset\) large enough, by \eqref{eq_aiv4viejiaJ2sae1quohNg1u}
\begin{equation}
\label{eq_xe0tu5eu0ubieMeiniengoog}
 \int_{\Omega} \abs{f_n} \abs{g_n - \varphi_\ell}
 \le t \int_{\Omega} \abs{g_n - \varphi_\ell}
 + {\norm{g_n-\varphi_\ell}_{\lebesgue^\infty\brk{\Omega, \Rset}}} \int_{\Omega}\brk{\abs{f_n} - t}_+
\end{equation}
and
\begin{equation}
\label{eq_eecah9niefai7Hohy8Xei3ph}
 \int_{\Omega} \abs{f} \abs{g - \varphi_\ell}
 \le t \int_{\Omega} \abs{g - \varphi_\ell}
 + \norm{g - \varphi_\ell}_{\lebesgue^\infty\brk{\Omega, \Rset}} \int_{\Omega}\brk{\abs{f} - t}_+
 \eqpunct.
\end{equation}
Combining \eqref{eq_xe0tu5eu0ubieMeiniengoog} and  \eqref{eq_eecah9niefai7Hohy8Xei3ph} with \eqref{eq_aiv4viejiaJ2sae1quohNg1u}, we get for \(n \in \Nset\) large enough
\[
\begin{split}
  &\abs[\bigg]{\int_{\Omega} f_n g_n - fg}\\
  &\qquad \le
  \abs[\bigg]{\int_{\Omega} \brk{f_n - f} \varphi_\ell}
  + t  \int_{\Omega} \brk{\abs{g_n - g} + 2 \abs{g - \varphi_\ell}}
  +  \varepsilon \brk{\norm{g_n-\varphi_\ell}_{\lebesgue^\infty\brk{\Omega, \Rset}} + \norm{g - \varphi_\ell}_{\lebesgue^\infty\brk{\Omega, \Rset}}} \eqpunct,
\end{split}
\]
we now choose \(\ell \in \Nset\) such that
\[
 2t \int_{\Omega} \abs{g - \varphi_\ell} \le \varepsilon
\]
and \(n_*\) such that if \(n \ge n_*\) we have
\[
  \abs[\bigg]{\int_{\Omega} \brk{f_n - f} \varphi_\ell}
  + t  \int_{\Omega} \abs{g_n - g} \le \varepsilon
  \eqpunct,
\]
and the conclusion then follows.
\end{proof}

\Cref{proposition_equiintegrable_limit_pullback} can now be obtained as a consequence of \cref{proposition_jacobian_measure} and \cref{lemma_distrib_to_L1_weak}.

\begin{proof}[Proof of \cref{proposition_equiintegrable_limit_pullback}]
We write
\[
 u^*\omega = \sum_{I \in \mathcal{I}} \brk{\omega_{I} \compose u} \sqb{\partial_{i_{1}} u \wedge \dotsb \wedge \partial_{i_{\ell}} u} \eqpunct{,}
\]
with
\[
 \mathcal{I} = \set{I = \brk{i_1, \dotsc, i_\ell}
 \in \set{1, \dotsc, m}^\ell \st i_1 < i_2 < \dotsb < i_\ell}
\]
and \(\omega_I \in \continuous \brk{\Rset^\nu, \operatorname{Lin}\brk{\bigwedge^{\ell}\Rset^\nu, \bigwedge^{\ell}\Rset^m}} \cap \lebesgue^\infty \brk{\Rset^\nu, \operatorname{Lin}\brk{\bigwedge^{\ell}\Rset^\nu, \bigwedge^{\ell}\Rset^m}}\).

For every \(I= \brk{i_1, \dotsc, i_\ell} \in \mathcal{I}\) and \(\varphi \in \smooth^{\infty}_c \brk{\Omega, \bigwedge^{m - \ell} \Rset^m}\), we have by \cref{proposition_jacobian_measure}
\[
  \lim_{n \to \infty} \int_{\Omega}
  \partial_{i_{1}} u_n \wedge \dotsb \wedge \partial_{i_{\ell}} u_n \wedge \varphi
 = \int_{\Omega} \partial_{i_{1}} u \wedge \dotsb \wedge \partial_{i_{\ell}} u \wedge \varphi
 \eqpunct.
\]
We also have for every \(n \in \Nset\) and \(t \in \intvo{0}{\infty}\),
\[
\begin{split}
 \int_{\Omega} \brk{\abs{\partial_{i_{1}} u_n \wedge \dotsb \wedge \partial_{i_{\ell}}u_n} - t}_+
 &\le \int_{\Omega} \brk{\abs{\Deriv u_n}^\ell - t}_+\\
 &\le 2^{\ell - 1} \int_{\Omega} \brk{\abs{\Deriv u_n} - t^{1/\ell}/2^{1 - 1/\ell}}_+^\ell \eqpunct{,}
\end{split}
\]
so that by \eqref{eq_EiKoevaiPazoo5sit7eiquaz}
\[
\lim_{t \to \infty} \limsup_{n \to \infty}
 \int_{\Omega} \brk{\abs{\partial_{i_{1}} u_n \wedge \dotsb \wedge \partial_{i_{\ell}}u_n} - t}_+
 = 0 \eqpunct .
\]
On the other hand, we have
\[
 \sup_{n \in \Nset} \, \norm{\omega_I \compose u_n}_{\lebesgue^\infty\brk{\Omega}} \le
 \norm{\omega}_{\lebesgue^\infty \brk{\Rset^\nu}} <\infty
\]
and by Lebesgue’s dominated convergence theorem,
since \(\omega_I\) is continuous,
\[
 \lim_{n \to \infty} \int_{\Omega}
 \abs{\omega_I \compose u_n - \omega_I \compose u}\abs{\varphi} = 0\eqpunct.
\]
Hence, \cref{lemma_distrib_to_L1_weak} is applicable and shows that
\[
  \lim_{n \to \infty} \int_{\Omega}\brk{\omega_I \compose u_n} \sqb{\partial_{i_{1}} u_n \wedge \dotsb \wedge \partial_{i_{\ell}} u_n} \wedge \varphi
 = \int_{\Omega} \brk{\omega_I \compose u} \sqb{\partial_{i_{1}} u \wedge \dotsb \wedge \partial_{i_{\ell}} u} \wedge \varphi
 \eqpunct,
\]
which proves the conclusion \eqref{eq_ief1Aef1umai2ahG7zo4ohh5}.
\end{proof}

We are now in a position to prove \cref{theorem_equi_integrable_cohomotology} as a consequence of \cref{proposition_equiintegrable_limit_pullback}.

\begin{proof}[Proof of \cref{theorem_equi_integrable_cohomotology}]
Let \(\brk{u_n}_{n \in \Nset}\) be a sequence in \(\smooth^{\infty}\brk{\manifold{M}, \manifold{N}}\) such that \(u_n \to u\) in \(\lebesgue^p \brk{\manifold{M}, \manifold{N}}\) and
\[
 \lim_{t \to \infty} \sup_{n \in \Nset} \int_{\manifold{M}} \brk{\abs{\Deriv u_n} - t}^p_+ = 0\eqpunct .
\]
Given \(\omega \in \smooth^\infty \brk{\manifold{N}, \bigwedge^{\floor{p}} \tangent^* \manifold{N}}\) such that \(\extdiff \omega = 0\), there exists \(\Bar{\omega} \in \smooth^\infty_c \brk{\Rset^\nu, \bigwedge^{\floor{p}} \Rset^\nu}\) such that
\(\Bar{\omega}\restr{\manifold{N}} = \omega\). One has then for every \(n \in \Nset\),
\[
\int_{\manifold{M}}  u_n^* \Bar{\omega} \wedge \extdiff \varphi = \int_{\manifold{M}} u_n^* \omega \wedge \extdiff \varphi = \brk{-1}^{\floor{p+1}}
 \int_{\manifold{M}} u_n^* \brk{\extdiff \omega} \wedge \varphi
 = 0
\]
and thus,
by \cref{proposition_equiintegrable_limit_pullback},
\[
 \int_{\manifold{M}} u^*  \omega \wedge \extdiff \varphi
 = \int_{\manifold{M}} u^* \Bar{\omega}\wedge \extdiff \varphi = \lim_{n \to \infty} \int_{\manifold{M}} u_n^* \Bar{\omega} \wedge \extdiff \varphi  = 0 \eqpunct,
\]
which proves our claim \eqref{eq_zohzuPaqu6lae1sahquo1Loo}.
\end{proof}

\end{document}